\documentclass[10pt]{article}
\title{Strongly quasiconvex functions: what we know (so far)}
\author{Sorin-Mihai Grad\thanks{Corresponding Author. Unit\'e de Math\'ematiques
Appliqu\'ees, ENSTA Paris, Institut Polytechnique de Paris, 91120 Palaiseau, France.
 E-mail: sorin-mihai.grad@ensta-paris.fr, ORCID-ID: 0000-0002-1139-7504} \and
Felipe Lara\thanks{Instituto de Alta investigaci\'on (IAI), Universidad de Tarapac\'a,
Arica, Chile. E-mail: felipelaraobreque@gmail.com; flarao@uta.cl. Web:
felipelara.cl, ORCID-ID: 0000-0002-9965-0921} \and
Ra\'ul T. Marcavillaca\thanks{Centro de Modelamiento Matem\'atico (CMM),
Universidad de Chile, Santiago, Chile. E-mail: raultm.rt@gmail.com; rtintaya@dim.uchile.cl,
ORCID-ID: 0000-0003-3748-0768}}

\usepackage{amsmath}
\usepackage{amsfonts}
\usepackage{amssymb}
\usepackage{graphicx}
\usepackage{color}
\usepackage{cancel}
\usepackage{ulem}
\usepackage{algorithm}
\usepackage{cite}
\usepackage[notcite,notref]{showkeys}
\usepackage{hyperref}

\numberwithin{equation}{section}

\newcommand{\norm}[1]{||#1||}

\newcommand{\R}{\mathbb{R}}%
\newcommand{\N}{\mathbb{N}}%

\DeclareMathOperator{\px}{Prox}
\DeclareMathOperator{\pr}{Pr}
\DeclareMathOperator{\dom}{dom}
\DeclareMathOperator{\amin}{\arg\min}
\DeclareMathOperator*{\argmin}{argmin}
\DeclareMathOperator{\id}{Id}
\DeclareMathOperator{\cl}{cl}
\DeclareMathOperator{\inte}{int}
\DeclareMathOperator{\fx}{Fix}

\newtheorem{theorem}{Theorem}

\newtheorem{corollary}[theorem]{Corollary}

\newtheorem{definition}[theorem]{Definition}
\newtheorem{example}[theorem]{Example}

\newtheorem{proposition}[theorem]{Proposition}
\newtheorem{remark}[theorem]{Remark}

\newenvironment{proof}[1][Proof]{\noindent \textbf{#1.} }{\  \rule{0.5em}{0.5em}}

\begin{document}

\maketitle

\begin{abstract}
\noindent {\bf Abstract.}
Introduced by Polyak in 1966, the class of strongly quasiconvex functions includes some interesting nonconvex members, like the square root of the Euclidean norm or ratios with a nonnegative strongly convex numerator and a concave and positive denominator. This survey collects the vast majority of the results involving strongly quasiconvex functions available in the literature at the moment, presenting, in particular, algorithms for minimizing such functions, and suggests some directions where additional investigations would be welcome.
\medskip

\end{abstract}

\medskip

\centerline{ {\it In loving memory of Boris Teodorovich Polyak.}}

\section{Introduction}

Introduced almost sixty years ago by Boris Polyak in the seminal paper \cite{POL} as a variety of the uniformly quasiconvex functions, the class of strongly quasiconvex functions includes some important nonconvex functions like the square root of the Euclidean norm or ratios with a nonnegative strongly convex numerator and a concave and positive denominator. Various properties of these functions were unearthed during these six decades by different authors in papers like \cite{JO1,KOR,LAR,VIA,VNC}, and, as anticipated by Polyak in \cite{POL}, the strongly quasiconvex functions play a role in the literature on calculus of variations as well (for instance, see \cite{KOS, IWK}). They turned thus to be more than a particular subclass of the quasiconvex functions (whose role in various applications, especially in economics and related fields, is well documented, see, for instance, \cite{D-1959,HAD,SZI,POG,TAN}).

In this survey we collected most of the known properties of the strongly quasiconvex functions and the vast majority of results involving them from the existing literature. Our work has been motivated by the recent developments in the literature, where this class of functions seems to be suitable for modeling different applications. Results known so far only for (strongly) convex functions were recently extended for strongly quasiconvex functions in directions where nothing of this kind has been done before, and this stresses the potential importance of this class of functions in the not so distant future in different fields of research and applications. Besides providing a state of the art in the literature on strongly quasiconvex functions to which some minor new results are added for completeness, a second goal of this survey is to ignite further investigations and employment of these functions.

Among the properties and results gathered in this paper we mention the existence of a unique minimizer of a strongly quasiconvex function over any closed convex set, the strong subdifferential that is a notion adapted to such functions and leads to a subgradient method for minimizing them, the convergence of the proximal point algorithm to the minimizer of a strongly quasiconvex function at a linear rate (that it thus available for convex functions that are strongly quasiconvex but not necessarily strongly convex, like the Euclidean norm), the convergence of trajectories of first-order and second-order dynamical systems to the minimizer of a differentiable strongly quasiconvex function as well as several algorithms for solving equilibrium problems involving bifunctions that are strongly quasiconvex in the second variable. Moreover, we also briefly discuss other notions labeled a strong quasiconvexity in the literature, in particular confusions between strictly quasiconvex functions or essentially quasiconvex functions and strongly quasiconvex ones. We also provided slight improvements and corrections to some results from the literature as well as a yet unknown relation with another class of strongly quasiconvex functions introduced in \cite{CFZ}, which was supposed to be different from  Polyak's one.

The paper is structured as follows. In Section \ref{sec:2} we introduce some necessary notations. In Section \ref{sec:3} we collected most of the existing properties and examples of strongly quasiconvex functions from the literature, and we also mention other usage of the name ``strongly quasiconvex''. Section \ref{sec:4} is dedicated to algorithms for minimizing strongly quasiconvex functions, and to continuous interpretations of such iterative methods via first-order and second-order dynamical systems. In Section \ref{sec:5} we gathered results involving equilibrium problems where the strong quasiconvexity in the second variable of the governing bifunctions plays a role, in particular iterative methods for solving such problems. Several ideas for future research are discussed in Section \ref{sec:6} as well as some ongoing works in different areas from continuous optimization and variational inequalities.

\section{Preliminaries}\label{sec:2}

For simplicity, all the results presented in this work are in finitely dimensional spaces, and we specify each time when in the original sources they are provided in more general settings. For similar preliminaries in Hilbert spaces we recommend the excellent book \cite{BCO}.

Consider the finitely dimensional Euclidean space $\R^n$, where all vectors we consider are column vectors, and denote its origin by $0_n$, and the corresponding \textit{Euclidean norm} by $\lVert \cdot \rVert$. Denote by $\id:{\R}^{n} \to {\R}^{n}$ the \textit{identity operator} on ${\R}^{n}$, and by
$K^{*} := \{y \in {\R}^{n}: \langle y, x\rangle \geq 0, ~ \forall ~
    x \in K\}$ the \textit{dual cone} associated to a set $K\subseteq {\R}^{n}$. The interior of a set $K\subseteq \R^n$ is denoted by $\inte K$, while its closure is $\cl K$. The closed ball centered in $x\in \R^n$ with radius $\delta>0$ is denoted by $\overline{\cal B} (x, \delta)$. Given a convex and closed set $K\subseteq \mathbb{R}^{n}$, the \textit{projection} of $x\in \mathbb{R}^{n}$ on $K$ is denoted by $\pr_{K} (x)$. Denote by ${\cal S}^n_+$ the cone of symmetric positive semidefinite real matrices, and by $e=(1, \hdots, 1)^\top \in {\R}^n$ the $n$-dimensional vector of ones. Given a matrix $A \in \mathbb{R}^{n \times n}$, we denote by $\sigma_{\min} (A)$ the smallest positive singular value of $A$.

Consider a function $h: {\R}^{n} \rightarrow \overline{{\R}} := {\R} \cup \{\pm \infty\}$. Its set of \textit{minimizers} is
$\amin_{{\R}^{n}} h$, and its \textit{effective domain} is $\dom\,h := \{x \in {\R}^{n}: h(x) < + \infty \}$. We call $h$ \textit{proper} if its domain is nonempty and $h$ takes nowhere the value
$- \infty$. The standard conventions $\sup_{\emptyset} h := - \infty$ and $\inf_{\emptyset} h := + \infty$ apply. The \textit{(strict) sublevel set of
$h$ at height} $t\in {\R}$ is $S_{t}^{(<)} (h) := \{x \in {\R}^{n}: h(x) \leq (<)t\}$.

When $\dom h$ is convex, we say that $h$ is
\begin{itemize}
 \item[$(a)$] \textit{convex}, when for every $x, y \in {\dom}\,h$
 \begin{equation}\label{def:convex}
  h(t x + (1-t)y) \leq t h(x) + (1 - t) h(y),
  \, \forall \, t \in [0, 1];
 \end{equation}

 \item[$(b)$] \textit{quasiconvex}, when for every $x, y \in {\dom}\,h$
  \begin{equation}\label{def:qcx}
   h(t x + (1-t) y) \leq \max \{h(x), h(y)\}, \, \forall
   \, t  \in [0, 1];
  \end{equation}

\item[$(c)$] \textit{semistrictly quasiconvex} if, given every
  $x, y \in {\dom}\,h$, with $h(x) \neq h(y)$, then
$$
   h(t x + (1-t)y) < \max \{h(x), h(y)\}, \, \forall \,
   x, y \in \dom h, \, \forall \, t \in \, ]0, 1[.
$$
\noindent When \eqref{def:convex} (or \eqref{def:qcx}) is strictly fulfilled whenever $x \neq y$ and $t \in \, ]0, 1[$, we call $h$ \textit{strictly
 (quasi)convex}. We also call $h$ \textit{(((semi)strictly) quasi)concave} when $-h$ is (((semi)strictly) quasi)convex.
\end{itemize}

Let $\emptyset\neq K\subseteq \dom h$. Then $h$ is said to be (see \cite{POL})
\begin{itemize}
 \item[$(a)$] \textit{strongly convex} with
 \textit{with modulus $\gamma> 0$} on $K$, when there is a $\gamma \in \, ]0,
 + \infty[$ for which
 $$ h(t y + (1-t)x) \leq t h(y) + (1-t) h(x) -
 t (1 - t) \frac{\gamma}{2} \lVert x - y \rVert^{2}, \,
 \forall \, x, y \in K, \, \forall \, t  \in [0, 1];$$

 \item[$(b)$] \textit{strongly quasiconvex}  with
 \textit{with modulus $\gamma> 0$} on $K$, when there is a $\gamma \in \,
 ]0, + \infty[$ for which
\begin{equation}\label{sqc}
 h(t y + (1-t)x) \leq \max \{h(y), h(x)\} - t(1 -
  t) \frac{\gamma}{2} \lVert x - y \rVert^{2}, \, \forall \, x, y \in
  K, \, \forall \, t  \in [0, 1].
\end{equation}
\noindent When $K=\dom h$ we do not mention ``on $K$'' when speaking about these notions. We call $h$ {\it strongly (quasi)concave} when $-h$ is strongly (quasi)convex.
 \end{itemize}

Clearly, strongly convex functions are strongly quasiconvex (with the same modulus). On the other hand, the classes of convex functions and strongly quasiconvex ones have a nonempty intersection (for instance, the strongly convex functions belong there), however no inclusion between them can be established. A linear function is convex but not strongly quasiconvex, while examples of nonconvex strongly quasiconvex functions can be found in Subsection \ref{su31}. It can be also shown that a strongly quasiconvex function is also strictly quasiconvex, however $t\in \R\mapsto 1/(1+|t|)$ is strictly quasiconvex but not strongly quasiconvex (see \cite{ILM}). The following scheme highlights the implications among the mentioned classes of (ge\-ne\-ra\-li\-zed) convex functions (where ``qcx'' stands for quasiconvex)
 
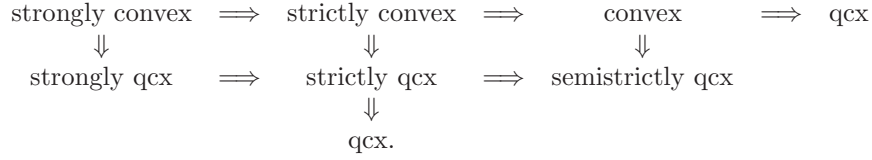
\begin{align*}
  \begin{array}{ccccccc}
  {\rm strongly ~ convex} & \Longrightarrow & {\rm strictly ~ convex} &
  \Longrightarrow & {\rm convex} & \Longrightarrow & {\rm qcx} \notag \\
  \Downarrow & \, & \Downarrow & \, & \Downarrow & \, & \, \notag \\
  {\rm strongly ~ qcx} & \Longrightarrow & {\rm strictly ~ qcx} &
  \Longrightarrow & {\rm semistrictly ~ qcx} & \, & \, \notag \\
  \, & \, & \Downarrow & \, & \,  & \, & \, \notag \\
  \, & \, & {\rm qcx}. & \, & \, & \, & \,
  \end{array}
 \end{align*}

A quasiconvex function can be considered to be a strongly quasiconvex one with modulus $\gamma = 0$. Different to the convex case, it has been shown in \cite[(1.4)]{LAR} that the sum of a quasiconvex and the half of a squared norm is not necessarily a strongly quasiconvex function, a counterexample ($x\in \R\mapsto x^3$) being provided in \cite[Remark 6]{LAR}. In order to simplify the presentation, we will refer to (strongly) quasiconvex functions throughout, even when the literature addresses (strongly) quasiconcave ones.

We employ the usual notations $h^\prime (\overline{x}, d)$ for the \textit{directional derivative} of a function $h: K \to \R$ (where $\emptyset\neq K \subseteq \R^n$) at $\overline{x}\in \R^n$ in direction $d\in \R^n$ and $\nabla h:K\to \R^n$ for the \textit{gradient} of $h$.
A differentiable function $h: \R^n\to \R$ satisfies the \textit{Polyak-\L ojasiewicz property} when there exists a $\nu >0$ such that $\| \nabla h(x)\|^2 \geq \nu (h(x)-h(\overline{x}))$ for all $x\in \R^n$, where $\overline{x}\in \argmin h$. 
A locally Lipschitz-continuous function $h: \R^n\to \R$ is said to be \textit{regular} when its directional derivative and its generalized directional derivative exist and coincide. A proper function $h: {\R}^{n} \rightarrow \overline{{\R}}$ is said to be \textit{$2$-supercoercive} when
$$\liminf_{\lVert x \rVert \rightarrow+ \infty} \frac{h(x)}{\lVert x \rVert^{2}} > 0.$$

Given a proper function $h: {\R}^{n} \rightarrow \overline{{\R}}$ and a closed and convex set $K\subseteq {\R}^{n}$ such that $K \subseteq \dom\,h$, the \textit{proximity operator of $h$ on $K$ of parameter $\beta > 0$} at $x \in {\R}^{n}$ is
$$\px_{\beta h}: {\R}^{n} \rightrightarrows {\R}^{n},
\ \px_{\beta h} (K, x) = \argmin_{y \in K} \left\{h(y) + \frac{1}{2 \beta}
\lVert y - x \rVert^{2} \right\}.$$
When $K = {\R}^{n}$, we only write $\px_{\beta h} (\cdot)$.

Like for almost all classes of nonconvex functions, the proximity operator of a strongly quasiconvex function is usually set-valued.

A more general notion of proximity operator can be defined by means of an asymmetric distance introduced below. Given an open and convex set $S\subseteq \R^n$, consider a differentiable function (see, for instance, \cite[Section 3.1]{LAM}) $\varphi: \cl S \to \R$ further called \textit{Bregman function with zone $S$} and define the \textit{Bregman distance associated to $\varphi$} as
$$D_{\varphi}:\cl S \times S \to \R,\ D_{\varphi}(x, y) := \varphi(x) - \varphi(y) -\nabla \varphi(y)^\top (x - y),$$
fulfilling the following properties
\begin{itemize}
\item $\varphi$ is strictly convex on $\cl S$;
\item $\varphi$ is continuously differentiable on $S$;
\item for any $x \in \cl S$ and any $t\in \R$ the partial sublevel set $S_{t} (D_{\varphi}(x,\cdot))$ is bounded;
\item for any $y \in S$ and any $t\in \R$ the partial sublevel set $S_{t} (D_{\varphi}(\cdot, y))$ is bounded;
\item if $\{x^k\}_k \subseteq \cl S$ is bounded, $\{y^k\}_k \subseteq S$ is such that $\lim_{k\to+\infty} y^k =\overline{y}\in \cl S$, and
$\lim_{k\to+\infty} D_{\varphi}(x^k, y^k)=0$, then $\lim_{k\to+\infty} x^k =\overline{y}$.
\end{itemize}
One has $D_{\varphi}(x, y)\geq 0$ for all $x, y \in \cl S \times S$ (due to the convexity of $\varphi$) with equality if and only if $x=y$.

Given $\varphi: \cl S \to \R$ a Bregman function with zone $S\subseteq \R^n$, the \textit{Bregman proximity operator of $h$ on $K$ of parameter $\beta > 0$} at $x \in S$ is
$$\px_{\beta h}^{\varphi}: S \rightrightarrows {\R}^{n},\ \px_{\beta h}^{\varphi} (K, x) = \argmin_{y \in K\cap \cl S} \left\{h(y) + \frac{1}{\beta}D_{\varphi}(y, x)\right\}.$$
When $S=\R^n$ and $D_{\varphi}(y, x)= (1/2)\lVert y - x \rVert^{2}$, $x, y\in \R^n$ the Bregman proximity operator of $h$ on $K$ of parameter $\beta > 0$ collapses to the proximity operator of $h$ on $K$ of parameter $\beta$.

A bifunction $f: {\R}^{n}\times {\R}^{n} \rightarrow {\R}$ is said to be \textit{monotone on $K \subseteq{\R}^{n}$} when
$$f(x, y) + f(y, x) \leq 0\ ~ \forall ~ x, y \in K,$$
and \textit{pseudomonotone on $K$} when
$$f(x, y) \geq 0 ~ \Longrightarrow ~ f(y, x) \leq 0  ~ \forall ~ x, y \in K.$$
As the names suggest, monotone bifunctions are also pseudomonotone, while the opposite implication is not always valid, see, for instance, the counterexamples in \cite{CMA,HKS}.

\section{Strongly quasiconvex functions}\label{sec:3}

\subsection{Examples of strongly quasiconvex functions}\label{su31}

Below we present examples of strongly quasiconvex functions that one can find in the literature. We begin with some functions that are convex and strongly quasiconvex but not strongly convex, followed by some strongly quasiconvex ones that are not even convex.

\begin{example}\label{ex1} (cf. \cite[Example 2]{JO1})
The function $t\in \R\mapsto |t+a|$, where $a\in \R$, is convex, and strongly quasiconvex with modulus $\gamma> 0$ on $[0, 1/ \gamma]$, however it is not strongly convex.
\end{example}

\begin{example}\label{ex2} (cf. \cite[Theorem 2 $\&$ Remark]{JO1})
The Euclidean norm is convex, and strongly quasiconvex with modulus $\gamma> 0$ on any bounded and convex set included in $\overline{\cal B}(0_n, 1/\gamma)$, however it is not strongly convex. Note, moreover, that it is not strongly quasiconvex on unbounded sets.
\end{example}

\begin{example}\label{ex7} (cf. \cite[Remark]{JO3})
The function $t\in \R\mapsto -t^2-t$ is strongly quasiconvex on $[0, 1]$, and prox-convex (see \cite{GL1}), but not convex.
\end{example}

\begin{example}\label{ex9} (cf. \cite[Example 4.9]{CGC})
For $c\in \R$, $d >0$ and $\delta > 0$, the function $t\in {\cal B} (0, \delta) \mapsto c - d \, e^{-x^2}$ is strongly quasiconvex on ${\cal B} (0, \delta)$ with modulus $\gamma = d \, e^{-\delta^2}$.
\end{example}

\begin{example}\label{ex8} (cf. \cite[Remark 9$(ii)$]{LMV})
The function $t\in \R\mapsto t^2 + 3\sin^2 t$ is differentiable and strongly quasiconvex, however not convex. Moreover, this function satisfies the Polyak-\L ojasiewicz property (see \cite{KNS}).
\end{example}

\begin{example}\label{ex3} (cf. \cite[Example 2]{JO2})
The function
$$f:[0, 1]\to \R,\ f(t) = \left\{
\begin{array}{ll}
0,& t=0\\
-\frac{1}{t},& t\in ]0, 1]
\end{array}
\right.$$ is strongly quasiconvex with modulus $\gamma = 1$ on $[0, 1]$, however it is neither convex not lower semicontinuous.
\end{example}

\begin{example}\label{ex5} (cf. \cite[Proposition 16]{LAR})
The function $t \in \R \mapsto \sqrt[4]{t^2+k^2}$, where $k\in \R$, is strongly quasiconvex with modulus $\gamma\in ]0, \frac{1}{2\sqrt[\frac{3}{4}]{c^2+k^2}}]$
on any interval $[-c, c]$, where $c>0$, but neither convex nor differentiable.
\end{example}

There are classes of functions used in signal/image recovery and machine learning (see, for instance, \cite{WCQ} for a great account) which are strongly quasiconvex, too.

\begin{example}\label{ex4} (cf. \cite[Corollary 3.6]{NAS}, \cite[Theorem 17 $\&$ Remark 19]{LAR})
The function $\lVert \cdot \rVert^{\alpha}$, where $0 < \alpha < 1$, is  strongly quasiconvex on nonempty bounded convex sets in $\R^n$. In particular, the function $x \in {\R}^{n} \mapsto \sqrt{\lVert x \rVert}$, which has been employed in studies on information protection (see \cite{VER}) or on Gram-Schmidt orthogonalization methods (see \cite{STU}), is strongly quasiconvex with modulus $\gamma\in ]0, \frac{1}{\sqrt[4]{80}\sqrt{r}}]$ on any bounded and convex subset of $\overline{\cal B}(0_n, r)$ but not convex. Note that this function is also not weakly convex, d.c. (difference of convex), prox-convex, or Lipschitz-continuous (see \cite{GL1, GL2, LMC}).
\end{example}

Some classes of fractional functions used in economic applications, see \cite{SCH,STA, CMA}) contain strongly quasiconvex that are not convex.

\begin{proposition}\label{pr1} (cf. \cite[Proposition 4.1]{ILM})
Let the set $\emptyset \neq K \subseteq {\R}^{n}$, and functions $h: \R^n \rightarrow \overline{{\R}}$ and $g: {\R}^{n} \rightarrow {\R}$ such that $K\cap \dom h \neq \emptyset$ and
$g(K)\subseteq ]0, +\infty[$. When $h$ is strongly convex with modulus $\gamma > 0$, $g$ is bounded from above by $M > 0$ on $K$, and one of the following hypotheses holds
\begin{enumerate}
\item[$(a)$] $g$ is affine;

\item[$(b)$] $h$ is nonnegative on $\dom h$ and $g$ is concave;

\item[$(c)$] $h$ is nonpositive on $\dom h$ and $g$ is convex;
 \end{enumerate}
then $h/g$ is strongly quasiconvex with modulus $\gamma^{\prime}= {\gamma}/{M} > 0$ on $K$.
\end{proposition}

A consequence of this statement follows.

\begin{corollary}\label{co1} (cf. \cite[Corollary 4.1]{ILM})
Let $A, B \in {\R}^{n\times n}$, $a, b \in {\R}^{n}$ and  $\alpha, \beta \in {\R}$.
Suppose that $A$ is a symmetric positive definite matrix, and $\lambda_{\min}(A)$ is its smallest eigenvalue.
Take $K = \{x \in {\R}^{n}: ~ m \leq (1/2) \langle Bx, x \rangle + \langle b, x \rangle + \beta \leq M\}$, with $0 < m < M$.
If any of the following conditions holds
 \begin{enumerate}
 \item[$(a)$] $B = 0\in \R^{n\times n}$;

 \item[$(b)$] $(1/2) \langle Ax, x \rangle + \langle a, x \rangle + \alpha \geq 0$ for all $x\in K$ and $B\in -{\cal S}^n_+$;

 \item[$(c)$] $(1/2) \langle Ax, x \rangle + \langle a, x \rangle + \alpha \leq 0$ for all $x\in K$ and $B\in {\cal S}^n_+$,
 \end{enumerate}
then the function defined as
$$x\in K\mapsto \frac{\frac{1}{2} \langle Ax, x \rangle + \langle a, x \rangle + \alpha}{\frac{1}{2} \langle Bx, x \rangle + \langle b, x \rangle + \beta}\in \R$$
is strongly quasiconvex with modulus $\gamma^{\prime} ={\lambda_{\min}(A)}/{M} > 0$ on $K$.
\end{corollary}

\subsection{Properties of strongly quasiconvex functions}\label{su32}

This subsection is dedicated to collecting and presenting the most important
properties of strongly quasiconvex functions and other interesting results
involving them.

We begin with the following statement on operations that preserve the strong
quasiconvexity of functions (including the quasiconvex case when $\gamma =
0$), from which part $(b)$ seems to be new.

\begin{proposition}\label{pr2}
 Let $A: \mathbb{R}^{n} \rightarrow \mathbb{R}^{n}$ be a linear operator,
 $h$ a strongly quasiconvex function with modulus $\gamma \geq 0$ and
 $\kappa > 0$. Then the following assertions hold
 \begin{itemize}
 \item[$(a)$] $\kappa h$ is strongly quasiconvex function with modulus
 $\kappa \gamma \geq 0$;

 \item[$(b)$] $f := h \circ A$ is strongly quasiconvex with modulus $\gamma \sigma_{\min} (A) \geq 0$;

 \item[$(c)$] If $h_{i}: \mathbb{R}^{n} \rightarrow \mathbb{R}^{n}$ is
 strongly quasiconvex functions with modulus $\gamma_{i} \geq 0$ for every
 $i \in \{1, \ldots, m\}$, then $h := \max_{1\leq i \leq m}\{h_{i}\}$ is strongly
 quasiconvex with modulus $\gamma := \min_{1\leq i \leq m} \{\gamma_{i}\} \geq 0$.
 \end{itemize}
\end{proposition}

\begin{proof}
 $(a)$: The statement is straightforward. $(c)$: See \cite{VNC}. We only prove $(b)$: Let $x_{1}, x_{2} \in \mathbb{R}^{n}$ and
 $t \in [0, 1]$. Then,
\begin{align*}
 f(t x_{1} + (1-t) x_{2}) & = h(A(t x_{1} + (1-t)
 x_{2})) \\
 & \leq \max\{h(A x_{1}), h(A x_{2})\} - t (1-t)  \frac{t}{2}
 \lVert A x_{1} - A x_{2} \rVert^{2} \\
 & = \max\{h(A x_{1}), h(A x_{2})\} - t (1-t)  \frac{t}{2}
 \langle A (x_{1} - x_{2}), A (x_{1} - x_{2}) \rangle \\
 & = \max\{h(A x_{1}), h(A x_{2})\} - t (1-t)  \frac{\gamma}{2}
 \langle A^{\top} A (x_{1} - x_{2}), x_{1} - x_{2} \rangle \\
 & \leq  \max\{ f(x_{1}), f(x_{2})\} - t (1-t)  \frac{\gamma \sigma_{\min} (A)}{2} \Vert x_{1} - x_{2} \rVert^{2},
\end{align*}
i.e., $f = h \circ A$  is strongly quasiconvex with modulus $\gamma \sigma_{\min} (A) \geq 0$.
\end{proof}

\begin{remark}
Using Proposition \ref{pr2} one can construct new (nonconvex) strongly quasiconvex functions starting from existing ones (for instance, those presented in Subsection \ref{su31}), as done in \cite[Remark 18]{LAR} or in \cite{GLM, GLT}.
\end{remark}

The next statement reveals that strongly quasiconvex functions defined on unbounded convex sets with a nonempty interior cannot be bounded from above, and it is used for providing a generalization of the classical statement that a quadratic function defined on $\R^n$ is quasiconvex if and only if it is convex.

\begin{proposition} (cf. \cite[Lemma]{JO3}) \label{pr20}
Let $\emptyset\neq K \subseteq \R^n$ be an unbounded convex set with a nonempty interior, and $h : K \to \R$ bounded from above. Then $h$ is not strongly quasiconvex.
\end{proposition}

\begin{proposition} (cf. \cite[Theorem 5]{JO3}) \label{pr21}
Let $\emptyset\neq K \subseteq \R^n$ be a convex cone with a nonempty interior, and $h : K \to \R$ a quadratic function. Then $h$ is strongly quasiconvex if and only if it is strongly convex.
\end{proposition}

\begin{remark}(cf. \cite[Remark]{JO3})
When $K$ is not a cone, the assertion of Proposition \ref{pr21} might fail to hold, take, for instance, $K=[0, 1]$ and $h$ the function in Example \ref{ex7}.
\end{remark}

The strong quasiconvexity of a real-valued function defined on a subset of $\R^n$ can be characterized by means of a the similar property of a real-valued function defined on a real interval.

\begin{proposition}\label{pr4}(cf. \cite[Theorem 1]{JO1})
Let $\emptyset\neq K \subseteq \R^n$ be a convex set, $h : K \to \R$, and $\gamma > 0$. Then $h$ is strongly quasiconvex with modulus $\gamma$ if and only if the function
$$t\in \R \mapsto h\left(x+\frac{t}{\|x-y\|}(x-y)\right)$$
is strongly quasiconvex on $[0, \|x-y\|]$ with the same modulus $\gamma$ for all $x, y\in K$ such that $x\neq y$.
\end{proposition}

Similarly to the strongly convex functions, the strongly quasiconvex functions that are lower se\-mi\-con\-ti\-nuous have exactly one minimizer on closed convex finitely-dimensional sets. As will be seen in Section \ref{sec:4}, this property allows extending proximal point methods from the convex setting to minimizing strongly quasiconvex functions where the generated sequences converge toward minimizers and not merely critical points. Before stating this result, we recall a more general one from which it is deduced. Note also that in \cite[Proposition 34]{KAL} the existence of a stronger type of minimizer of a strongly quasiconvex function is highlighted.

\begin{theorem} (cf. \cite[Theorem 1]{LAR}\label{th0})
Let $\emptyset\neq K \subseteq \R^n$ be a convex set, and $h : K \to \R$ a strongly quasiconvex function with modulus $\gamma > 0$.
Then, $h$ is $2$-supercoercive (in particular, coercive).
\end{theorem}

\begin{remark}
 One can rediscover \cite[Theorem 2]{JO2} and \cite[Theorem 3]{JO2} as
 consequences of Theorem \ref{th0}.
\end{remark}

\begin{remark}
In \cite[Proposition 2.1]{ILM} one can find hypotheses guaranteeing the coercivity (but probably not the $2$-supercoercivity) of a strongly quasiconvex function defined on a Hilbert space.
\end{remark}

\begin{theorem} (cf. \cite[Corollary 3]{LAR})\label{th2}
Let $\emptyset\neq K \subseteq \R^n$ be a closed and convex set, and $h : K \to \R$ a strongly quasiconvex function with modulus $\gamma > 0$ that is also lower semicontinuous.
Then, $h$ has exactly one minimizer on $K$.
\end{theorem}

\begin{remark}
 If $\emptyset\neq K \subseteq \R^n$ is a closed and convex set, $h : K \to \R$ a lower semicontinuous strongly quasiconvex function with modulus $\gamma > 0$ and $\beta > 0$, then  $\px_{\beta h} (K,\cdot) \neq \emptyset$, but it is not necessarily a singleton, see \cite[Remark 6]{LAR}.
\end{remark}

The minimal value of a strongly quasiconvex function satisfies a quadratic growth property.

\begin{proposition} (cf. \cite[Theorem 1]{KOR}, see also \cite[Proposition 34]{KAL}) \label{pr9}
Let $\emptyset\neq K \subseteq \R^n$ be a convex set, and $h : K \to \R$ a strongly quasiconvex function with modulus $\gamma \geq 0$ that attains its minimum on $K$ at $\overline{x}\in K$. Then, for any point $x \in K$, the following inequality holds
$$h(x)-h(\overline{x})\geq \frac{\gamma}{8} \|x-\overline{x}\|^2.$$
\end{proposition}

When $h$ is strongly quasiconvex, the fact that $\px_{\beta h} (\mathbb{R}^{n},\cdot)$ should be a singleton (for all $\beta > 0$) is equivalent to the convexity of $h$, as highlighted in the following statement.

\begin{proposition}\label{pr3}(cf. \cite[Proposition 14]{LAR})
Let $\emptyset \neq K\subseteq \R^n$ be convex, and $h : K \to \R$. The following statements are equivalent
\begin{itemize}
\item [(a)] $h$ is convex;
\item [(b)] the function $h + \frac{1}{2\beta} \|z-\cdot\|^2$ is strongly convex for all $z \in \R^n$ and all $\beta > 0$;
\item [(c)]the function $h + \frac{1}{2\beta} \|z-\cdot\|^2$ is quasiconvex for all $z \in \R^n$ and all $\beta > 0$;
\item [(d)] the function $h + \frac{1}{2\beta} \|z-\cdot\|^2$ is strongly quasiconvex for all $z \in \R^n$ and all $\beta > 0$.
\end{itemize}
\end{proposition}

As the (convex) subdifferential is often empty for nonconvex functions, and the ``mainstream'' nonsmooth ones do not seem to have specific properties in the case of generalized convex functions (as far as we are aware, the only known results concerning the Clarke subdifferential of strongly quasiconvex functions are available in \cite{DUO, VIA}, see Corollary \ref{co2} and Propositions \ref{pr17}-\ref{pr18} below), subdifferential notions specific to quasiconvex functions were proposed in the literature (see \cite{GRP, PEZ}). In the recent work \cite{KAL}, one can find a subdifferential specific to strongly quasiconvex functions. Below we briefly present it together with some of its properties. Other results involving it can be found in \cite[Proposition 40, Corollary 41, Proposition 42]{KAL}.

\begin{definition}\label{de1} (see \cite[Definition 5 $\&$ Definition 19]{KAL})
Let $\emptyset\neq K \subseteq \R^n$, $h:\R^n\to \overline {\R}$ be proper, $\beta >0$ and  $\gamma \geq 0$. The $(\beta, \gamma, K)$-strong subdifferential of $h$ at $\overline{x}\in K\cap \dom h$ is
\begin{eqnarray*}
\partial ^K_{\beta, \gamma} h(\overline{x}) &=& \Big\{z\in \R^n: \max \{h(y), h(\overline{x})\}\geq h(\overline{x})+ \frac{t}{\beta} z^\top (y-\overline{x})\\
 &+& \frac{t}{2} \left(\gamma - \frac{t}{\beta} - t\gamma\right) \|y-\overline{x}\|^2 \ \forall y\in K\ \ \forall t\in [0, 1]\Big\}.
\end{eqnarray*}
When $K=S_{h(\overline{x})}(h)$, we call $\partial_{\beta, \gamma} h(\overline{x}):=\partial ^{S_{h(\overline{x})}(h)}_{\beta, \gamma} h(\overline{x})$ the $(\beta, \gamma)$-strong sublevel subdifferential of $h$ at $\overline{x}$.
\end{definition}

\begin{remark}(see \cite{KAL})
 If $\gamma > 0$, then $\partial^{K}_{\beta, \gamma} h(\overline{x})$ is useful
 for dealing with strongly quasiconvex functions, while if $\gamma = 0$, then
 $\partial^{K}_{\beta, 0} h(\overline{x})$ is an alternative subdifferential for the
 class of quasiconvex functions. Note that when $x \in \dom h$ and $S_{h(x)} (h)
 \subseteq K$,
\begin{equation}\label{inclusion}
   \partial^{K}_{\beta, \gamma} h(x) \, \subseteq \, \partial_{\beta, \gamma}
   h(x), ~ \forall \beta > 0, ~ \forall \gamma \geq 0.
  \end{equation}
\end{remark}

\begin{remark}
In \cite[Remark 6]{KAL} one finds simple examples (like the identity function on $\R$) showing that the $(\beta, \gamma, K)$-strong subdifferential (for $\emptyset\neq K \subseteq \R^n$, $\beta >0$ and  $\gamma \geq 0$) does not coincide with the standard (convex) subdifferential even for continuously differentiable convex functions. Moreover, extending the function in Example \ref{ex3} to the whole space by assigning it the value $+\infty$ outside the interval $[0, 1]$, one has at hand a situation where most of the subdifferentials specific for quasiconvex functions are empty at $0$, while (for $\beta >0$) one has $\partial ^\R_{\beta, 1} h(0) = ]-\infty, -\beta/2]$.
\end{remark}

\begin{proposition} (cf. \cite[Corollary 38]{KAL}) \label{pr14}
Let $\emptyset\neq K \subseteq \R^n$ be a closed convex set, and $h : K\to \overline\R$ strongly quasiconvex function with modulus $\gamma > 0$ on $K$. Then $\partial ^K_{\beta, \gamma} h (x)$ is nonempty, closed and convex for all $x\in K$.
\end{proposition}

\begin{proposition} (cf. \cite[Proposition 40]{KAL}) \label{pro:KLx}
	Let $K \subseteq \mathbb{R}^n$ be a closed and convex set, $h:
	\mathbb{R}^{n} \rightarrow \overline{\mathbb{R}}$  be a proper and
	lsc function such that $K \subseteq \mathrm{dom}\,h$, $\beta>0$ and
	$\overline{x}\in K$. If $h$ is strongly quasiconvex on $K$ with modulus
	$\gamma>0$, then
	\begin{align}
	y \in K \cap S_{h(\overline{x})} (h) \, \Longrightarrow \, \langle \xi, y
	- \overline{x} \rangle \leq - \frac{\beta \gamma}{2} \lVert y - \overline{x}
	\rVert^2, ~ \forall \xi \in \partial_{\beta, \gamma}^K h(\overline{x}).\notag
\end{align}
\end{proposition}

\begin{proposition} (cf. \cite[Proposition 3.3]{LMC}) \label{pr11}
Let $\emptyset\neq K \subseteq \R^n$ be a closed convex set, and $h : \R^n\to \overline\R$ a proper, continuous and strongly quasiconvex function with modulus $\gamma > 0$ on $K$, with $K\subseteq \dom h$. Let $\{x^k\}_k\in K$ and $z^k\in \partial ^K_{\beta, \gamma} h(x^k)$ for $k\in \N$. If $x^k\to \overline{x}\in K$ and $z^k\to 0_n$, then $0_n\in \partial ^K_{\beta, \gamma} h (\overline{x})$. Hence, $\overline{x}$ is a global minimizer of $h$.
\end{proposition}

\begin{proposition} (cf. \cite[Proposition 3.4]{LMC}) \label{pr12}
Let $\emptyset\neq K \subseteq \R^n$ be a closed convex set, and $h : \R^n\to \overline\R$ a proper, continuous and strongly quasiconvex function with modulus $\gamma > 0$ on $K$, with $K\subseteq \inte \dom h$. If $K$ is compact, then $\bigcup_{x\in K} \partial ^K_{\beta, \gamma} h (x)$ is nonempty and bounded.
\end{proposition}

\begin{remark}
A relation between the convex subdifferential and the strong one for classes of strongly quasiconvex functions presented in Proposition \ref{pr1} (quadratic fractional frunctions) can be found in \cite[Proposition 4.2]{LMY}.
\end{remark}

One of the latest developments involving strongly quasiconvex functions concerns algorithms for minimizing them. Besides the subgradient ones (where the strong subdifferential plays a crucial role), the proximal point method and some of its variations have been extended from the convex setting to minimizing strongly quasiconvex functions. The properties of the proximity operator of a strongly quasiconvex function known at the moment are listed below. For other technical results involving the  proximity operator of strongly quasiconvex functions we refer to \cite{KAL,LAR}.

\begin{proposition} (cf. \cite[Proposition 9]{LAR}) \label{pr5}
 Let $\emptyset\neq K \subseteq \R^n$ be a closed and convex set, $h : K \to \R$ a lower semicontinuous and strongly quasiconvex
 function with modulus $\gamma>0$, and $\beta>0$. Then
$$ \fx \left( \px_{\beta h}(K,\cdot)\right) = \amin_{K} h.$$
\end{proposition}

\begin{remark}\label{re1} (cf. \cite[Corollary 5]{LAR}) The proximity operator of a strongly quasiconvex function is usually set-valued, turning into a single-valued mapping when the sum of the strongly quasiconvex function and a half of a squared norm (possibly multiplied by a positive constant) is strongly quasiconvex, too.
An example where the sum of two strongly quasiconvex functions is strongly quasiconvex, a situation which does not happen in general, can be found in \cite[Example 4.1]{ILM} (see also \cite[Remark 4.2]{ILM}). Note also that Proposition \ref{pr3} provides a restriction on this assumption.
\end{remark}

A partial extension of Proposition \ref{pr5} for Bregman proximity operators is provided below (see also \cite[Remark 3.4]{LAM}).

\begin{proposition} (cf. \cite[Proposition 3.3]{LAM}) \label{pr23}
Let $K \subseteq \R^n$ be a closed and convex set with a nonempty interior, $h : K \to \R$ a lower semicontinuous and strongly quasiconvex
 function with modulus $\gamma>0$, and $\beta>0$. Then
$$ \fx \left( \px^{\varphi}_{\beta h}(K,\cdot)\right) \supseteq \amin_{K} h.$$
\end{proposition}

The next result is provided in \cite{LAR} for $x\in K\subseteq \R^n$, but it can be extended to the whole space without problems.

\begin{proposition} (cf. \cite[Proposition 7]{LAR}) \label{pr6}
Let $\emptyset\neq K \subseteq \R^n$ be a closed and convex set, $h : K \to \R$ a lower semicontinuous and strongly
 quasiconvex function with modulus $\gamma \geq 0$, $\beta> 0$, and $x \in \R^n$. If $\overline{x} \in \px_{\beta h} (K, x)$, then
 for all $y \in K$ and all $t\in [0, 1]$ one has
 $$\max \{h(y), h(\overline{x})\} \geq h(\overline{x}) + \frac{t}{
  \beta}  (x - \overline{x})^{\top} (y - \overline{x}) + \frac{t}{2}
  \left(\gamma - \frac{t}{\beta} - t \gamma \right) \lVert y -
  \overline{x} \rVert^{2}.$$
\end{proposition}

\begin{proposition} (cf. \cite[Proposition 36]{KAL}) \label{pr13}
Let $\emptyset\neq K \subseteq \R^n$ be a closed and convex set, $h : K \to \R$ a lower semicontinuous and strongly quasiconvex function with modulus $\gamma \geq 0$, $\beta> 0$ and $\overline{x}, z \in K$. Then
$$\overline{x} \in \px_{\beta h} (K, z) \Longrightarrow z-\overline{x} \in \partial ^K_{\beta, \gamma} h(\overline{x}).$$
\end{proposition}

\begin{remark}
The reverse implication in Proposition \ref{pr13} is an open problem.
\end{remark}

\begin{remark}
In our work \cite{GLH} (see also \cite{ILM}), we extended to Hilbert spaces Proposition \ref{pr1}, Proposition \ref{pr9}, Proposition \ref{pr6}, Proposition \ref{pr5} and Remark \ref{re1}. The reader is invited to verify whether other similar results involving strongly quasiconvex functions (presented here or not) can be extended from finitely-dimensional spaces to that setting, too. We also refer to the very recent preprint \cite{NAS} for
more results involving strongly quasiconvex functions in infinitely dimensional spaces.
\end{remark}

Now, we present some results involving strongly quasiconvex functions having some differentiability properties. We begin with a first-order characterization of the strong quasiconvexity of $C^1$ functions, followed by some consequences and other related statements.

\begin{proposition} (cf. \cite[Theorem 1]{JO3}, originally from \cite[Theorems 2 $\&$ 6]{VNC}) \label{pr19}
Let $\emptyset\neq K \subseteq \R^n$ be a convex set, $\gamma > 0$, and $h : K \to \R$ a continuously differentiable function. Then $h$ is a strongly quasiconvex function with modulus $\gamma$ if and only if
$$h(x)\leq h(y) \ \Longrightarrow \ \nabla h(y)^\top (y-x)\geq \frac{\gamma}{2} \|x-y\|^2, \ \forall x, y\in K.$$
\end{proposition}

The following two corollaries are easy consequences of Proposition \ref{pr19}.

\begin{corollary}\label{co3} (cf. \cite[Corollary 5]{LMV})
Let $\emptyset\neq K \subseteq \R^n$ be a convex set, $\gamma \geq 0$, and $h : K \to \R$ a differentiable and strongly quasiconvex function with modulus $\gamma$.
If for every $x, y \in K$ one has $\nabla h(x)^\top (x-y) \geq 0$ whenever $x\in S_{h(y)}(h)$, then $h$ is strongly convex with modulus $\gamma$.
\end{corollary}

\begin{corollary}\label{co4} (cf. \cite[Corollary 6]{LMV})
Let $\emptyset\neq K \subseteq \R^n$ be a convex set, $\gamma \geq 0$, and $h : K \to \R$ a differentiable function. Then $h$ is strongly quasiconvex with modulus $\gamma$ if and only if $\nabla h(x)\in \partial_{1, \gamma} h(x)$ for all $x\in K\cap \dom h$.
\end{corollary}


\begin{proposition} (cf. \cite[relation (4)]{KOR}) \label{pr7}
Let $\emptyset\neq K \subseteq \R^n$ be a convex set, and $h : K \to \R$ a strongly quasiconvex function with modulus $\gamma > 0$ that is differentiable in all directions at $\overline{x}\in K$. Then, for any point $x \in K$ such that $h(x) \leq h(\overline{x})$, the following inequality holds
$$h^\prime (\overline{x}, x-\overline{x})\leq -\frac{\gamma}{2} \|x-\overline{x}\|^2.$$
\end{proposition}

\begin{proposition} (cf. \cite[Theorem 2]{KOR}) \label{pr8}
Let $\emptyset\neq K \subseteq \R^n$ be a convex set, and $h : K \to \R$ a differentiable and strongly quasiconvex function with modulus $\gamma \geq 0$, whose gradient is Lipschitz-continuous with a constant $L>0$, and
which attains its minimum on $K$ at $\overline{x}\in K$. Then, Polyak-\L ojasiewicz property holds with value $\frac{2L}{\gamma^2}$, that is,
$$\|\nabla h(x)\|^2 \geq \frac{2L}{\gamma^2} (h(x) - h(\overline{x})), ~ \forall x \in K.$$
\end{proposition}

\begin{remark} (cf. \cite[Remark 9]{LMV})
The converse statement in Proposition \ref{pr8} does not hold in general as the example in \cite[Example 4.1.3]{Nesterov-book} shows.
\end{remark}

Some results involving the \textit{Clarke subdifferential} of the considered strongly quasiconvex function (denoted by $\partial^C$) follow. The first one is a consequence of Proposition \ref{pr7}. In \cite[Lemma 2 \& Corollary 1]{VIA} one can find additional technical results in the same direction.

\begin{corollary}\label{co2} (cf. \cite[Corollary 2.1]{DUO})
Let $\emptyset\neq K \subseteq \R^n$ be an open convex set, $h : K \to \R$ a strongly quasiconvex function with modulus $\gamma > 0$ that is locally Lipschitz-continuous and regular on $K$, some set $\emptyset\neq D \subseteq K$, and $\overline{x}\in D$. Then
\begin{enumerate}
\item[$(a)$] if $0_n \in \partial^C h(\overline{x})$, then $\amin_{D} h =\{\overline{x}\}$;

\item[$(b)$] for $h$ to take the smallest (on $D$) value at $\overline{x} \in \inte D$, it is necessary and sufficient that $0_n \in \partial^C h(\overline{x})$.
\end{enumerate}
\end{corollary}

\begin{remark}
When $K=D=\R^n$, the statement in Corollary \ref{co2}$(a)$ is shown without imposing the regularity of $h$ in \cite[Lemma 1]{VIA}.
\end{remark}

\begin{proposition} (cf. \cite[Theorem 2]{VIA}) \label{pr17}
Let $h : \R^n \to \R$ be strongly quasiconvex function with modulus $\gamma > 0$ and locally Lipschitz-continuous. Then, for all $x, y\in \R^n$, one has
$$h(y) \leq h(x) \ \Longrightarrow \ z^\top
(y-x) \leq - \frac{\gamma}{2} \|x-y\|^2, ~\forall z \in \partial^C h(x).$$
\end{proposition}

\begin{remark}
In virtue of Proposition \ref{pr19}, the validity of the reverse statement in
Proposition \ref{pr17} is an open problem.
\end{remark}

\begin{proposition} (cf. \cite[Theorem 3]{VIA}) \label{pr18}
Let $\emptyset\neq K \subseteq \R^n$ be an open set, and $h : K \to \R$ a strongly quasiconvex function with modulus $\gamma > 0$ that is locally Lipschitz-continuous. If there exists a constant $\beta >0$ such that $\|z\|\leq \beta$ for all $x\in K$ and all $z\in \partial^C h(x)$, then there exists a constant $\sigma>0$ such that $\|x-y\|\leq \sigma \|p-q\|$ for all pairs of vectors $(x, p), (y, q)$ with $x, y\in K$, $h(x)=h(y)$, and $p, q$ unit outward normal vectors to the level set $S_{h(x)}(h)$ at $x$, respectively $y$.
\end{proposition}

\begin{remark}
A further development of the statement in Proposition \ref{pr8} can be found in \cite[Theorem 5]{KOR}. In the same paper one can find a short proof showing that a strongly quasiconvex function directionally differentiable on an open convex set is also strictly pseudoconvex, while in \cite[Proposition 15]{LAR} the same conclusion is derived for a
strongly quasiconvex function differentiable on a convex set. On the other hand, a technical result involving strongly quasiconvex functions that are locally Lipschitz-continuous and regular is provided in \cite[Lemma 2.1]{DUO} and the employed for deriving necessary and sufficient optimality conditions for characterizing minimizers of strongly quasiconvex functions over weakly convex sets by means of the Clarke subdifferential can be found in \cite[Theorems 3.1--3.3 $\&$ Corollary 3.1]{DUO}. Moreover, in \cite[Theorem 4.3]{TAN} one can find a technical result involving strongly quasiconvex locally Lipschitz-continuous production functions that states the local Lipschitz-continuity of the conditional factor demand functions in cost minimization models.
\end{remark}

\begin{remark}
Strongly quasiconvex functions are also employed in defining the class of strongly $G$-subdifferentiable on ($K \subseteq \R^n$, taken closed and convex) functions, that
are those $h: \R^n \rightarrow \overline{\mathbb{R}}$ proper and lower semicontinuous for which $K \cap \dom \,h \neq \emptyset$ that are strongly quasiconvex on $K$ and fulfill
$$\forall z\in K, \exists~\overline{x}\in \R^n:\ \px_{h} (K, z) = \{\overline{x}\}\ \text{ and }\
  \frac{1}{2} (z - \overline{x}) \in \partial^{\leq}_{K} h(\overline{x}),$$
where
$$
 \partial^{\leq} h(x) := \left\{ \xi \in \mathbb{R}^{n}: ~ h(y) \geq h(x) +
 \xi^{\top} (y - x), ~ \forall y \in S_{h(x)} (h)\right\},
$$
when $x \in \dom \,h$, and $\partial^{\leq} h(x) = \emptyset$ if $x \not\in \dom \,h$, is the \textit{Guti\'errez subdifferential} of $h$ at $x\in \R^n$ (see \cite{GUT}). In \cite[Proposition 3.7]{GL1} it is shown that strongly $G$-subdifferentiable functions on $K$ are prox-convex (see \cite[Definition 3.1]{GL1} for this generalized convexity notion), while \cite[Corollary 3.1]{GL1} displays a direct implication between strongly quasiconvex functions and prox-convex ones under certain hypotheses involving the Guti\'errez subdifferential. Moreover, in \cite[Proposition 3.8]{GL1} connections between strongly quasiconvex functions and positively quasiconvex ones (see \cite[Definition 3.3$(c)$]{GL1}) are established.
\end{remark}

\subsection{Locally strongly quasiconvex functions and other similar notions}\label{su33}

It is also important to address the fact that the literature on strong quasiconvexity is far from being consistent with respect to definitions. As mentioned above, in order to simplify the presentation, we only speak of strongly quasiconvex functions, even when, especially in economic applications, strongly quasiconcave ones were considered.

Generalizations of the notion of strongly quasiconvex functions can be found, for instance, in \cite{ANS} (strongly $\eta$-quasiconvex functions, extended to higher-order strongly $\eta$-quasiconvex ones in \cite{RKH}) or \cite{POL} (where one also finds uniformly quasiconvex functions and other related notions). Moreover, in \cite{CGC} one finds several local and global notions for bifunctions based on strong quasiconvexity. Worth mentioning is also the strong quasiconvexity notion for vector functions proposed in \cite{DGW} (that is not really an extension of its scalar counterpart as defined in this survey).

On the other hand, one should also be aware that different types of functions that do not fulfill \eqref{sqc} are called strongly quasiconvex in various works. Often (see, e.g. \cite{TIW}) \textit{strictly} quasiconvex functions are labeled as \textit{strongly} quasiconvex. Other authors (see, for instance, \cite{SWI}) call strongly quasiconvex functions that are both quasiconvex and semistrictly quasiconvex, which are usually known as \textit{essentially quasiconvex} in the modern literature. Further, in \cite{SNK} functions satisfying a quadratic growth condition are labeled as \textit{strongly quasiconvex}, while in \cite{PEW, HEF, DAZ} other (different) types of so-called strongly quasiconvex functions are considered. Moreover, in \cite{GIO, HIR} other definitions of strongly quasiconvex $C^2$ functions are considered. A (higher order) strong quasiconvexity notion defined by means of the Clarke subdifferential for locally Lipschitz-continuous functions can be found in \cite{BAK}. In \cite{SIL} a type of strongly quasiconvex functions are defined via integrals, while another one, that is usually considered in calculus of variations can be found, for instance, in \cite{KOS}, while in \cite{IWK} the latter definition is both particularized and generalized. Last but not least, in \cite{TTZ} strongly quasiconvex sequences are considered and employed in defining another type of strongly quasiconvex functions. See also \cite{POG} for other usages of the name strongly quasiconvex functions in the literature.

A rather intriguing case is the class of weakly strongly quasiconvex functions, introduced in \cite{CFZ}, which was designed with the intention of better capturing some targeted properties and was thought to differ from the one introduced in Polyak's definition \eqref{sqc}. Based on it, a class of functions called in \cite{CFZ} \textit{strongly quasiconvex} was introduced, as follows.

\begin{definition}\label{CFZ:strongqcx} (cf. \cite[Definitions 3.1
 and 3.3]{CFZ})
 Let $\emptyset \neq K \subseteq \mathbb{R}^{n}$ and $h: K
 \rightarrow \mathbb{R}$ be a function. It is said that $h$ is
 $\alpha$-quasiconvex at $\overline{x} \in K$ (with $\alpha \in
 \mathbb{R}$) if there exist $\rho > 0$ and $e \in \mathbb{R}^{n}$,
 with $\lVert e \rVert = 1$, such that
 \begin{equation}\label{alpha:qcx}
  y \in K \cap {\cal B} (\overline{x}, \rho) \cap S_{h(\overline{x})}
 (h) ~ \Longrightarrow ~ \langle e, y - \overline{x} \rangle \geq \alpha
  \lVert y - \overline{x} \rVert^{2}.
 \end{equation}
We say that $h$ is \textit{CFZ-strongly quasiconvex on $K$} if for every
 $\overline{x} \in K$, there exists $\alpha(\overline{x}) > 0$ such that
 $h$ is $\alpha(\overline{x})$-quasiconvex on $S_{h(\overline{x})} (h)$.
\end{definition}

Although the original motivation in \cite[page 999]{CFZ} was to provide a different definition for strong quasiconvexity (and then to obtain an existence result for this new notion), we reveal below a yet unknown (as far as we are aware) connection between these classes of functions but in a local sense. First we deal with the smooth case.

\begin{proposition}\label{P:implies:CFZ}
 Let $K \subseteq \mathbb{R}^{n}$ be an open convex set and $h: K
 \rightarrow \mathbb{R}$ a continuously differentiable function.
 If $h$ is strongly quasiconvex with modulus
 $\gamma > 0$, then $h$ is CFZ-strongly quasiconvex.
\end{proposition}

\begin{proof}
 Let $\overline{x} \in K$
 and $y \in S_{h(\overline{x})} (h)$.
 If $\nabla h(\overline{x}) = 0$, then $\overline{x}$ is the unique minimizer
 of $h$ on $K$ by Theorem \ref{th2}, i.e., relation \eqref{alpha:qcx}
 follows immediately.

 If $\nabla h(\overline{x}) \neq 0$, then since $h$ is
 strongly quasiconvex with modulus $\gamma > 0$, it follows from Proposition \ref{pr19} that
 \begin{align*}
  h(y) \leq h(\overline{x}) & \, \Longrightarrow ~ \langle \nabla h
  (\overline{x}), y - \overline{x} \rangle \leq -\frac{\gamma}{2} \lVert
  y - \overline{x} \rVert^{2} \\
  & \Longleftrightarrow \, \left\langle - \frac{\nabla h(\overline{x})}{
  \lVert \nabla h (\overline{x}) \rVert}, y - \overline{x} \right\rangle
  \geq \frac{\gamma}{2 \lVert \nabla h(\overline{x}) \rVert} \lVert y -
  \overline{x} \rVert^{2}.
 \end{align*}

 Since $\gamma > 0$, by taking $e = - (1/\| \nabla h(\overline{x}) \|)\nabla h(\overline{x})
$ and $\alpha(\overline{x}) = \gamma / (2
 \lVert \xi (\overline{x}) \rVert) > 0$, we obtain that $h$ is
 $\gamma/(2 \lVert \nabla h(\overline{x}) \rVert)$-quasiconvex at
 $\overline{x} \in K$, i.e., $h$ is CFZ-strongly quasiconvex.
\end{proof}

Following the same idea, we can prove a similar results for any lower semicontinuous strongly quasiconvex function by employing the strong subdifferential.

\begin{theorem}\label{theo:PimpliesCFZ}
 Let $K \subseteq \mathbb{R}^{n}$ be a closed and convex set and
 $h: K \rightarrow \mathbb{R}$ be a lower semicontinuous strongly quasiconvex function with modulus $\gamma > 0$. Then $h$ is
 CFZ-strongly quasiconvex.
\end{theorem}

\begin{proof}
 Since $h$ is strongly quasiconvex with modulus $\gamma > 0$, it follows from Proposition \ref{pr14} that $\partial^{K}_{\beta, \gamma} h(x) \neq \emptyset$ for all $x \in K$ and all $\beta > 0$.

 Let $\overline{x} \in K$. Since $h$ is lower semicontinuous and strongly quasiconvex with
 modulus $\gamma > 0$, we have for every $\beta > 0$ that (see Proposition \ref{pro:KLx})
\begin{align}\label{nonsmooth:ineq}
  y \in K \cap S_{h(\overline{x})} (h) & \Longrightarrow ~ \langle \xi, y
  - \overline{x} \rangle \leq - \frac{\gamma \beta}{2} \lVert y -
  \overline{x} \rVert^{2}\ ~ \forall ~ \xi \in \partial^{K}_{\beta,
  \gamma} h(\overline{x}).
 \end{align}
 If $\overline{x} \in {\amin}_{K}\,h$, then $K \cap
 S_{h(\overline{x})} (h) = \{\overline{x}\}$, thus $y=\overline{x}$ and
 there is nothing to prove. If $\overline{x} \notin {\amin}_{K}\,h$,
 then there exists $\xi(\overline{x}) \in \partial^{K}_{\beta,  \gamma}
 h(\overline{x})$, with $\xi(\overline{x}) \neq 0$, such that from
 relation \eqref{nonsmooth:ineq} we obtain that
 \begin{align*}
  & \langle \xi(\overline{x}), y - \overline{x} \rangle \leq - \frac{\gamma
  \beta}{2} \lVert y - \overline{x} \rVert^{2} ~ \Longleftrightarrow ~
  \left \langle - \frac{\xi(\overline{x})}{\lVert \xi(\overline{x}) \rVert},
  y - \overline{x} \right \rangle \geq \frac{\gamma \beta}{2 \lVert
  \xi(\overline{x}) \rVert} \lVert y - \overline{x} \rVert^{2}.
 \end{align*}

 Since $\gamma > 0$, by taking $e = - (1/\|\xi(\overline{x}) \|) \xi(\overline{x})$ and $\alpha(\overline{x}) = (\gamma
 \beta)/(2 \| \xi(\overline{x}) \|)$, $h$ is $\alpha
 (\overline{x})$-quasiconvex at $\overline{x}$, and since $\overline{x}
 \in K$ was arbitrary, the function $h$ is CFZ-strongly quasiconvex.
\end{proof}

The reverse statements in Proposition \ref{P:implies:CFZ} and Theorem \ref{theo:PimpliesCFZ} do not hold in general, as we show in the following remark.

\begin{remark}\label{rem:not-equiv}
 Let $h: [0, +\infty[ \, \rightarrow \mathbb{R}$ be the function
 given by $h(x) = x$. Clearly, $h$ is not strongly quasiconvex because it is not $2$-supercoercive (see Theorem \ref{th0}).

  On the other hand, let $\overline{x} \in [0, + \infty[$ and $y \in
  S_{h(\overline{x})} (h)$, i.e., $y \leq \overline{x}$. If $\overline{x}
  = 0$, then $y = \overline{x} = 0$, i.e, there is nothing to prove.
  Furthermore, if $\overline{x} > 0$ and $y=\overline{x}$, then there is
  nothing to prove again. Thus, suppose that $\overline{x} > 0$ and that
  $y < \overline{x}$, thus $0 < \overline{x}-y < \overline{x}$, i.e.,
  $-1/(\overline{x}-y) < -1/\overline{x}$. Since $h^{\prime}
  (x) = 1$ for all $x \in [0, + \infty[$, we have
  \begin{align*}
   & \langle h^{\prime} (\overline{x}), y - \overline{x} \rangle = -
   (\overline{x} - y) = - \frac{1}{(\overline{x} - y)} (\overline{x}
   - y)^{2} \leq - \frac{1}{\overline{x}} (\overline{x} - y)^{2} \\
   & \hspace{2.0cm} \Longleftrightarrow ~ \langle - h^{\prime}
   (\overline{x}), y - \overline{x} \rangle \geq \frac{1}{2} \left( \frac{2
   }{\overline{x}} \right) (\overline{x} - y)^{2}.
  \end{align*}
  Hence by taking $e= - h^{\prime} (\overline{x})$ and
  $\alpha(\overline{x}) = 2/\overline{x} > 0$, we have that $h$
  is $2/\overline{x}$-quasiconvex at $\overline{x}$. Therefore,
 $h$ is CFZ-strongly quasiconvex on $[0, + \infty[$.
\end{remark}

Theorem \ref{theo:PimpliesCFZ} basically shows that CFZ-strong quasiconvexity is actually a {\it local version} of the strong quasiconvexity as defined by Polyak, so, a more appropriate name for this notion could probably be {\it local strong quasiconvexity}. We refer to \cite[Corollary 3.6]{CFZ} for other properties of the functions belonging to this class, for which stronger versions (involving strongly quasiconvex functions) are available in \cite[Corollary 1]{VIA} and \cite[Corollary 3.2]{GL1}.

\medskip

Another class of generalized convex functions whose name resembles the one of the strongly quasiconvex ones is the one of \textit{quasi strongly convex functions} introduced in \cite{NNG} (see also \cite{ZSJ} for accelerated algorithms involving such functions). For $\gamma >0$, a differentiable function $h: \R^n \rightarrow \mathbb{R}$ is said to be \textit{$\gamma$-quasi strongly convex} when for all $x\in \R^n$ one has
$$\nabla h(x)^\top \left(x-\pr_{\argmin h}(x)\right) \geq h(x)-h(\pr_{\argmin h}(x))+\frac{\gamma}{2} \|x-\pr_{\argmin h}(x)\|^2.$$
When $h\in C^1$, \cite[Proposition 7]{LMV} asserts that the $\gamma$-quasi strongly convex functions that admit unique minimizers are strongly quasiconvex with modulus $\gamma > 0$.

Last but not least let us recall another class of functions whose definition resembles the one of the strongly quasiconvex ones, namely the \textit{sharply quasiconvex functions with modulus $\gamma \geq 0$} (cf. \cite{KKO}), that are differentiable functions $h: K \rightarrow \R$ defined on a convex set $\emptyset \neq K \subseteq \R^n$ for which whenever $x,y\in K$ and all $t\in [0, 1]$ one has
$$\nabla h(y)^\top(x-y)\geq 0 \Rightarrow h(t y + (1-t)x) \leq \max \{h(y), h(x)\} - t(1 -t) \frac{\gamma}{2} \lVert x - y \rVert^{2}.$$
From the construction it is clear that strongly quasiconvex functions are also sharply quasiconvex. A scheme connecting these notions and the strong convexity with (generalized) monotonicity properties of their gradients can be found in \cite[Proposition 15]{LMV}.

\section{Algorithms for minimizing strongly quasiconvex functions}\label{sec:4}

Let $\emptyset \neq K \subseteq \mathbb{R}^{n}$ be a nonempty closed and convex set, and $h: K \rightarrow \mathbb{R}$ a continuous and strongly quasiconvex function with modulus $\gamma > 0$. Consider the constrained optimization problem
\begin{equation}\label{COP}
 \min_{x \in K} h(x). \tag{COP}
\end{equation}

The first proximal point method for solving \eqref{COP} was proposed by Lara in \cite{LAR}, and subsequently developed into a Bregman type proximal point method in \cite{LAM} and relaxed-inertial proximal point algorithms in \cite{GLM, GLH, GNV}. On the other hand, \cite{LMC} presents a subgradient projection method for solving \eqref{COP}. Moreover, one can find in the recent literature (see\cite{LMV, ROP, RPK}) dynamical systems whose trajectories converge to the optimal solution to \eqref{COP} when the strongly quasiconvex function to be minimized is (twice) continuously differentiable.

\subsection{Algorithms for strongly quasiconvex functions}\label{su41}

The first algorithm we mention was first proposed for $K$ in a Hilbert space and $h$ a convex function in \cite{ATC} and is based on the methods proposed in \cite{ALZ,GOM,MAI}, while in our paper \cite{GLM} one finds likely its first usage for nonconvex functions. Slightly more general than in \cite{GLM}, we take $K$ to be an affine subspace of $\R^n$ below.

\begin{algorithm}[H] (see \cite[Algorithm 1]{GLM})
\caption{RIPPA for Strongly Quasiconvex Functions (RIPPA-SQ)}\label{rippa:sqcx}
\begin{description}
 \item[Step 0.] (Initialization). Let  $x^0 = x^{-1} \in K$, $\alpha \in
 [0, 1[$, $0 < \rho^{\prime} \leq \rho^{\prime \prime} < 2$, $\{c_{j}\}_j \subseteq ]0, +\infty[$ and $k = 0$.

 \item[Step 1.] Choose $\alpha_k\in [0,\alpha]$ and set 
 $$
  y^{k} = x^{k} + \alpha_{k} (x^{k} - x^{k-1}), \qquad [\text{extrapolation step}]
 $$
 and compute
 \begin{equation}\label{prox:sqcx}
  z^{k} \in \px_{c_{k} h} (K, y^{k}) \qquad ~~~~~ [\text{proximal step}].
 \end{equation}
 \item[Step 2.] If $z^{k} = y^{k}$: \textsc{stop}.

\item[Step 3.] Choose $\rho_k \in [\rho^{\prime}, \rho^{\prime
 \prime}]$ and update
 \begin{equation}\label{eq:relax.step}
    x^{k+1}=(1-\rho_k)y^k+\rho_k z^{k} \qquad [\text{relaxation step}].
 \end{equation}
 \item[Step 4.] Let $k = k+1$ and go to Step 1.
 \end{description}
\end{algorithm}

\begin{remark}\label{re3}
When one has $z^{k} = y^{k}$ for some $k\in \N$, \cite[Lemma 2.3]{GLM} yields that $y^k\in K$ is a fixed point of the proximity operator of $h$ on $K$ of parameter $c_k$. Then Proposition \ref{pr5} guarantees that $y^{k}$ is also an optimal solution to \eqref{COP}.
\end{remark}

The convergence statement for Algorithm \ref{rippa:sqcx} follows. In it, the following set plays a role (where $\{z^{k}\}_{k}$ is generated in \eqref{prox:sqcx})
\begin{equation}\label{omeg}
\Omega := \{x \in K: \, h(x) \leq h(z^{k}), ~ \forall ~ k \in \N\}.
\end{equation}
As $h$ is strongly quasiconvex, Theorem \ref{th2} yields $\Omega \neq \emptyset$. Note that when $h$ is merely quasiconvex, the convergence statements for Algorithm \ref{rippa:sqcx} (see \cite[Proposition 3.3 \& Proposition 4.2]{GLM}, also \cite[Section 3.4]{LAM} and \cite[Lemma 3.2 and Theorem 3.1]{MAI}) require imposing $\Omega \neq \emptyset$.

\begin{theorem} (cf. \cite[Theorem 3.1]{GLM}) \label{th1}
Let $K \subseteq \mathbb{R}^{n}$ be an affine subspace, $h: K \rightarrow \mathbb{R}$ a continuous and strongly quasiconvex function with modulus $\gamma > 0$,
$0 < \rho^{\prime} \leq \rho^{\prime \prime} < 2$, $\{\rho_{k}\}_{k} \subseteq [\rho^{\prime}, \rho^{\prime \prime}]$, $\alpha \in [0, 1[$, $\{\alpha_{k}\}_{k}
 \subseteq [0, \alpha]$ and $\{x^{k}\}_{k}$, $\{y^{k}\}_{k}$ and $\{z^{k}\}_{k}$,
 be the sequences generated by Algorithm \ref{rippa:sqcx}. If
 \begin{align}\label{eq:th:main.01}
  \sum_{k=0}^\infty \, \alpha_{k} \norm{x^{k} - x^{k-1}}^{2} < + \infty,
  \end{align}
 then the following assertions hold
 \begin{itemize}
  \item [$(a)$] for every $x^{\ast} \in \Omega$, the limit $\lim_{k \to \infty}
  \lVert x^{k} - x^{\ast} \rVert$ exists and
  \begin{align}\label{eq:6}
   \lim_{k \to + \infty} \lVert x^{k+1} - y^{k} \lVert = \lim_{k \to +
   \infty} \lVert z^{k} - y^{k} \lVert = 0;
  \end{align}
  \item [$(b)$] if, in addition, $c_{k} \geq c^{\prime} > 0$ for every $k \geq 0$,
  then the sequence $\{x^{k}\}_{k}$ converges to $\overline{x}$, where $\amin_{K}\,h = \{\overline{x}\}$, and $\lim_{k \to + \infty} h(x^k) = h(\overline x)= \min_{K} h$; moreover, the
  sequences $\{y^{k}\}_{k}$ and $\{z^{k}\}_{k}$ converge both to $\overline{x}$, too.
  \end{itemize}
\end{theorem}

\begin{remark}\label{re4}
In a forthcoming work \cite{GNV} several enhancements to Algorithm \ref{rippa:sqcx} and Theorem \ref{th1} are provided. Besides extending the usability of this iterative method to Hilbert spaces, modifications that allow its employment for minimizing strongly quasiconvex functions over closed convex sets (while maintaining its convergence) are proposed. One of them consists in replacing $y^k$ with $x^k$ in the relaxation step \eqref{eq:relax.step} and asking $\rho_k$ to lie between 0 and 1 (this guarantees that the new iterate $x^{k+1}$ is feasible to \eqref{COP}), while the other one relaxes (from a computational point of view) the proximal step \eqref{prox:sqcx} to $z^{k} \in \px_{c_{k} h} (y^{k})$ and adds a projection onto $K$ in the relaxation step \eqref{eq:relax.step} which becomes $x^{k+1}=\pr_K((1-\rho_k)y^k+\rho_k z^{k})$. The convergence statements for both these variants of Algorithm \ref{rippa:sqcx} are similar to Theorem \ref{th1}, with the notable difference that the continuity imposed on $h$ in \cite{GLM} is relaxed to lower semicontinuity.
\end{remark}

Sufficient conditions for guaranteeing the fulfillment of \eqref{eq:th:main.01} (that is a standard hypothesis in the literature on inertial proximal point methods, see \cite{ALA, ALZ}), inspired by \cite[Proposition 2.5]{ALZ}, were provided in \cite[Theorem 4.1 and Corollary 4.1]{GLM} (see also \cite[Remark 4.2]{GLM}). For completeness, we present below the main result of this type.

\begin{proposition}\label{pr22} (cf. \cite[Theorem 4.1]{GLM})
Let $K \subseteq \mathbb{R}^{n}$ be an affine subspace, $h: K \rightarrow \mathbb{R}$ a continuous and strongly quasiconvex function with modulus $\gamma > 0$, $0 < \rho^{\prime} \leq \rho^{\prime \prime}<2$, $\{\rho_{k}\}_{k}
 \subseteq [\rho^{\prime}, \rho^{\prime \prime}]$, $\alpha \in [0, 1[$,
 $\{\alpha_{k}\}_{k} \subseteq [0, \alpha]$ and $\{\alpha_k\}_{k}$ is
 nondecreasing satisfying (for some $\beta > 0$)
 \begin{align}\label{eq:alpha}
  0 \leq \alpha_{k} \leq \alpha_{k + 1} \leq \alpha < \beta < 1, ~ \forall
  ~ k \geq 0,
 \end{align}
 and
$$ 
  \rho^{\prime \prime} = \rho^{\prime \prime} (\beta, \rho^{\prime}) :=
  \dfrac{2 \rho^{\prime}(\beta^{2} - \beta + 1)}{2 \rho^{\prime} \beta^{2}
  + (2- \rho^{\prime}) \beta + \rho^{\prime}}.$$
If $\{x^{k}\}_{k}$ is the sequence generated by Algorithm \ref{rippa:sqcx}, then
$$
 \sum_{k=1}^\infty\norm{x^k-x^{k-1}}^2< +\infty.
$$
\end{proposition}

\begin{remark}\label{re5}
When $\alpha = 0$ it immediately follows $\alpha_{k} = 0$ for all $k \in \N$, and Algorithm \ref{rippa:sqcx} collapses into a relaxed proximal point algorithm (see \cite[Corollary 3.1]{GLM}) where $K$ can be taken only closed and convex (see also Remark \ref{re4}) when $\rho_{k} \in \, ]0, 1]$ for every $k \in \N$ (or an affine subspace when $\rho_{k} \in \, ]0, 2[$, $k \in \N$).
\end{remark}

\begin{remark}\label{re6}
When besides $\alpha = 0$ one also fixes $\rho_{k} = 1$ for all $k \in \N$, Algorithm \ref{rippa:sqcx} turns into the proximal point method for solving \eqref{COP} (see \cite[Corollary 3.2 \& Remark 3.4]{GLM} and \cite[Theorem 10 $\&$ Remark 13$(iv)$]{LAR}), for whose convergence it is enough to take $h$ lower semicontinuous instead of continuous. Note, moreover, that, additionally to the results from the literature mentioned above, also Theorem \ref{th1} yields the convergence of the sequence generated by the proximal point method to the minimizer of the involved strongly quasiconvex function.
\end{remark}

\begin{remark}\label{re2}
In \cite[Section 5]{GLM} it is shown that a suitable choice of the relaxation and inertial parameters ensures that Algorithm \ref{rippa:sqcx} solves \eqref{COP} faster and cheaper than its standard proximal point counterpart. As noted in \cite{GLM}, the optimal choices of the involved relaxation and inertial parameters known in the convex case (see \cite{ATC}) do not actually accelerate the proximal point algorithm in the strongly quasiconvex case, at least in the situations considered there. How to determine which values of the inertial and relaxation parameters provide a theoretically guaranteed acceleration in the strongly quasiconvex setting remains an open question.
\end{remark}

Another development of the proximal point method for solving \eqref{COP} \cite[Algorithm 1]{LAR} is the Bregman proximal point algorithm from \cite[Algorithm 1]{LAM}. After providing some technical results \cite[Propositions 3.1--3.4 \& Remark 3.5]{LAM}, the following iterative method for solving \eqref{COP} is proposed. Note that for the corresponding convergence statements it is enough to take $h$ lower semicontinuous (instead of continuous), while, on the other hand, $K$ 
needs to have a nonempty interior.

\begin{algorithm}[H] (see \cite[Algorithm 1]{LAM})
\caption{Bregman PPA for Strongly Quasiconvex Functions (BPPA-SQ)}\label{bppa:sqcx}
\begin{description}
 \item[Step 0.] (Initialization). Let $\varphi$ a Bregman function with zone $S$ such that $K\cap \cl S\neq \emptyset$, $\{c_{j}\}_j \subseteq ]0, +\infty[$, $x^0\in S$ and $k = 0$.

 \item[Step 1.] Compute
 $$ x^{k+1} \in \px^{\varphi}_{c_{k} h} (K, x^{k}) \qquad ~~~~~ [\text{proximal step}].
 $$
\item[Step 2.] If $x^{k+1}$ is an optimal solution to \eqref{COP}: \textsc{stop}.

\item[Step 3.] Let $k = k+1$ and go to Step 1.
 \end{description}
\end{algorithm}

\begin{remark}
According to \cite[Remark 3.4]{LAM}, an alternative stopping criterion that could be employed in Algorithm \ref{bppa:sqcx} is $x^{k+1} = x^{k}$. When activated, it yields that $x^k$ is the optimal solution to \eqref{COP} (see \cite[Proposition 3.3]{LAM}.
\end{remark}

\begin{remark}
The proximal point algorithm for solving \eqref{COP} \cite[Algorithm 1]{LAR} (see also Remark \ref{re6}) can be derived as a special case of Algorithm \ref{bppa:sqcx} (see \cite[Corollary 3.1]{LAM}), too.
\end{remark}

Before providing the corresponding convergence statement we need to introduce a set similar to $\Omega$ (see \eqref{omeg}), that is
$\Omega^B := \{x \in K: \, h(x) \leq h(x^{k}), ~ \forall ~ k \in \N\}$.

\begin{theorem} (cf. \cite[Proposition 3.5 $\&$ Theorem 3.1]{LAM}) \label{th3}
Let $K \subseteq \mathbb{R}^{n}$ be a closed and convex set with a nonempty interior, $h: K \rightarrow \mathbb{R}$ a lower continuous and strongly quasiconvex function with modulus $\gamma > 0$, $\varphi$ a Bregman function with zone $S\subseteq \R^n$ such that $K\cap \cl S\neq \emptyset$, $\{c_{k}\}_{k \in \mathbb{N}}$ be a sequence of positive numbers, and $\{x^{k}\}_{k}$ the sequence generated by Algorithm \ref{bppa:sqcx}. If $\{x^{k}\}_{k}\subseteq S$, then the following assertions hold
 \begin{itemize}
  \item [$(a)$] if $x^{k+1} \neq x^k$, then $h(x^{k+1}) < h(x^k)$;
  \item [$(b)$] the sequence $\{x^{k}\}_{k}$ is bounded;
  \item [$(c)$] for all $x\in \Omega^B$, the sequence $\{D_{\varphi}(x^k, x)\}_k$ is convergent;
  \item [$(d)$] if, in addition, $c_{k} \geq c^{\prime} > 0$ for every $k \geq 0$, then $\{x^{k}\}_{k}$ is a minimizing sequence of $h$, i.e. $\lim_{k \to + \infty} h(x^k) = \min_{K} h$.
  \end{itemize}
\end{theorem}

\begin{remark}\label{re12}
Another proximal point type algorithm for solving \eqref{COP} was proposed in \cite[Corollary 3.2]{ILP} by specializing an iterative method for solving equilibrium problems involving bifunctions that are strongly quasiconvex in the second variable. As its intermediate steps consist in solving equilibrium problems we do not include it here and more about it can be seen in Remark \ref{re11}.
\end{remark}

Last but not least we present the subgradient projection method for solving \eqref{COP} proposed in \cite{LMC} that makes use of the $(\beta, \gamma, K)$-strong subdifferential of the involved function $h$ (for $\beta >0$ and  $\gamma \geq 0$) (recall Definition \ref{de1}). Again, we skip the technical results \cite[Corolaries 3.1--3.2 \& Proposition 3.5]{LMC}, focusing on the iterative method and its convergence statement.

\begin{algorithm}[H] (see \cite[Algorithm 1]{LMC})
\caption{Subgradient Method for Strongly Quasiconvex Functions (SM-SQ)}\label{sm:sqcx}
\begin{description}
 \item[Step 0.] (Initialization). Let $\beta,\gamma > 0$, $\{\alpha_j\}_j\subseteq \left]0,\frac{1}{\gamma\beta}\right[$, $x^0\in K$ and $k = 0$.

 \item[Step 1.] If $0_n\in \partial ^K_{\beta, \gamma} h (x^k)$: \textsc{stop}.

\item[Step 2.] Compute
 $$ \xi^{k} \in \partial ^K_{\beta, \gamma} h (x^k)\qquad ~~~~~ [\text{subgradient step}].
 $$

\item[Step 3.] Let $x^{k+1} = \pr_K(x^k-\alpha_k \xi^k)$\qquad ~~~~~ [\text{projection step}].

\item[Step 4.] Let $k = k+1$ and go to Step 1.
 \end{description}
\end{algorithm}

\begin{remark} (see\cite[Lemma 2.3 $\&$ Remark 3.3]{LMC}
When the stopping criterion in Algorithm \ref{sm:sqcx} $0_n\in \partial ^K_{\beta, \gamma} h (x^k)$ is activated, then $x^{k+1}$ is the optimal solution to \eqref{COP}. Moreover, an alternative stopping criterion for Algorithm \ref{sm:sqcx} could be $x^{k+1} = x^{k}$, that, too, yields that $x^{k+1}$ is the optimal solution to \eqref{COP}.
\end{remark}

\begin{theorem} (cf. \cite[Theorem 3.1]{LMC}) \label{th4}
Let $K \subseteq \mathbb{R}^{n}$ be a closed and convex set, $h: K \rightarrow \mathbb{R}$ a lower continuous and strongly quasiconvex function with modulus $\gamma > 0$ such that $K\subseteq \inte \dom h$, and $\beta,\gamma > 0$ for which there exists an $M > 0$ such that $\|\xi\|\leq M$ whenever $\xi\in \partial ^K_{\beta, \gamma} h(x)$ for all $x\in K$. Take the sequence $\{\alpha_j\}_j\subseteq ]0,1/(\gamma\beta)[$ such that $\sum_{k=0}^{+\infty} \alpha_k = +\infty$ and $\sum_{k=0}^{+\infty} \alpha_k^2 < +\infty$, and let $\{x^{k}\}_{k}$ be the sequence generated by Algorithm \ref{sm:sqcx}. Then $\{x^{k}\}_{k}$ converges to $\overline{x}$, where $\amin_{K}\,h = \{\overline{x}\}$, and $\lim_{k \to + \infty} h(x^k) = \min_{K} h$. Moreover, when $h$ is continuous, then $\lim_{k \to + \infty} h(x^k) = h(\overline x)= \min_{K} h$.
\end{theorem}

\begin{remark}\label{re14}
Results on convergence rates of the iterative methods mentioned above for solving \eqref{COP} can be found in \cite[Remark 3.7]{LAM}, \cite[Corollary 3.3]{LMC}, \cite[Section 5.1]{GLH} and in \cite{GNV}, the last two in Hilbert spaces. Different to the convex case, the proximal point method turns out to converge linearly for strongly quasiconvex functions and so do its relaxed-inertial versions mentioned above. This provides thus guarantees of linear convergence for the proximal point type methods when minimizing convex functions that are also strongly quasiconvex (but not necessarily strongly convex), like the Euclidean norm (see Example \ref{ex2}).
\end{remark}

Other algorithms for solving \eqref{COP} are provided in the next subsection for the case $K=\R^n$ and $h\in C^1$.

\subsection{Dynamical systems involving strongly quasiconvex functions}\label{su42}

Consider further that $h: \R^n \rightarrow \mathbb{R}$ (i.e. $K=\R^n$) is strongly quasiconvex with modulus $\gamma> 0$ and also continuously differentiable. In \cite{LMV, ROP, RPK} one can find dynamical systems whose trajectories converge to the optimal solution to \eqref{COP} that are also discretized, providing gradient type algorithms for solving \eqref{COP}. Recall that Theorem \ref{th2} yields that \eqref{COP} has a unique optimal solution, let us denote it by $\overline{x}\in \R^n$.

We begin with the first-order dynamical system considered in \cite{ROP, LMV}, that is
\begin{equation}\label{DS1}
  \dot{u}(t) = \nabla h(u(t)) + \psi (t),\ t\in [0, +\infty[,\ u(0)=x^0\in \R^n,
\end{equation}
where $\psi \in W^{1,1}(]0,+\infty[, \R^n)$. In \cite{ROP} \eqref{DS1} is studied in Hilbert spaces (and one actually obtains strong convergence statements), while in \cite{LMV} in Euclidean ones and for $\psi\equiv 0$. Worth noticing is also that \cite{ROP} is mainly devoted to the case when $h$ is quasiconvex (with the strong quasiconvexity hypothesis added only a few times to obtain stronger outcomes), while in \cite{LMV} it is strongly quasiconvex.

The Cauchy-Lipschitz theorem yields the existence of a unique solution $\overline{u}\in C^1([0, +\infty), \R^n)$ to \eqref{DS1} when $\nabla h$ is Lipschitz-continuous, however we will not impose this hypothesis in the whole subsection because \cite[Section 4]{LMV} explicitly specifies that the results provided there do not require it. Besides various technical results where $h$ is not necessarily strongly quasiconvex, \cite{ROP} contains the following convergence statement for the trajectories of \eqref{DS1}, whose boundedness hypothesis can be guaranteed, for instance, when $\liminf_{t\to +\infty}\|\overline{u}(t)\| < +\infty$ (see \cite[Proposition 2.3]{ROP}).

\begin{theorem} (cf. \cite[Theorem 2.8]{ROP}) \label{th10}
When $\nabla h$ is Lipschitz-continuous and $\overline{u}\in C^1([0, +\infty), \R^n)$ is bounded, $\liminf_{t\to +\infty}\overline{u}(t) = \overline{x}$.
\end{theorem}

On the other hand, the convergence statement from \cite{LMV} does not involve the uniqueness of the trajectory of \eqref{DS1} as no Lipschitz-continuity is imposed on $\nabla h$, while, on the other hand, it concerns a simpler dynamical system, as $\psi\equiv 0$.

\begin{theorem} (cf. \cite[Theorem 18]{LMV}) \label{th11}
When $\psi\equiv 0$, for any trajectory $\hat{u}\in C^1([0, +\infty), \R^n)$ of \eqref{DS1} the function $t\in ]0, +\infty[ \to h(\hat{u}(t))$ is nonincreasing, and
$\liminf_{t\to +\infty}\hat{u}(t) = \overline{x}$.
\end{theorem}

Discretizing \eqref{DS1} with respect to time, one obtains the following gradient descent algorithm for minimizing $h$.

\begin{algorithm}[H] (see \cite[(3.1)]{ROP} and \cite[Algorithm 1]{LMV})
\caption{Gradient Method for Strongly Quasiconvex Functions (GM-SQ)}\label{gm:sqcx}
\begin{description}
 \item[Step 0.] (Initialization). Let $\alpha>0$, $\{\alpha_j\}_j\subseteq \left[\alpha,+\infty\right[$, $\{\psi^j\}_j\in \ell^1$, $x^0\in \R^n$ and $k = 0$.

 \item[Step 1.] If $\nabla h(x^k)=0_n$: \textsc{stop}.

\item[Step 2.] Compute
 $$ x^{k+1} = x^k  + \alpha_k \nabla h (x^k) + \psi^k\qquad ~~~~~ [\text{gradient step}].
 $$

\item[Step 3.] Let $k = k+1$ and go to Step 1.
 \end{description}
\end{algorithm}

Again, the corresponding statement from \cite{ROP} reveals the (strong, in Hilbert spaces) convergence of the sequence generated by Algorithm \ref{gm:sqcx} toward the unique minimizer of $h$.

\begin{theorem} (cf. \cite[Theorem 2.8]{ROP}) \label{th17}
When $\nabla h$ is Lipschitz-continuous, and the sequence $\{x^{k}\}_{k}$ generated by Algorithm \ref{gm:sqcx} is bounded, then $\{x^{k}\}_{k}$ converges to $\overline{x}$.
\end{theorem}

\begin{remark}\label{re13}
\cite[Theorem 2.8]{ROP} actually contains an additional hypothesis, that basically is $\Omega^B\neq \emptyset$. Due to Theorem \ref{th2}, we know that this is an outcome of the assumptions currently imposed on $h$ (even in Hilbert spaces, see \cite[Theorem 3.1]{GLT}). Therefore, the open problem posed on \cite[page 10]{ROP} is solved to the affirmative.
\end{remark}

Different to the continuous case, the convergence statement for Algorithm \ref{gm:sqcx} from \cite{LMV} follows after several technical results that we do not recall here and requires the additional hypothesis of local Lipschitz-continuity on the gradient of $h$, that yields the Lipschitz-continuity of $\nabla h$ with a constant $\bar L>0$ on
$S_{h(x^0)}(h)$.

\begin{theorem} (cf. \cite[Theorem 22]{LMV}) \label{th18}
When $\psi^k=0_{\ell^1}$ for all $k\in \N$, $\nabla h$ is locally Lipschitz-continuous, and there exist $\alpha, \bar \alpha>0$ such that $\alpha \leq \alpha_k \leq \bar \alpha < \min \{\gamma/\bar L^2, 2/\bar L\}$, then the sequence $\{x^{k}\}_{k}$ generated by Algorithm \ref{gm:sqcx} converges to $\overline{x}$.
\end{theorem}

Further we discuss the second-order dynamical systems proposed in \cite{LMV} and \cite{RPK} for solving \eqref{COP}, namely
\begin{equation}\label{DS2}
  \ddot{u}(t) + \alpha \dot{u}(t) + \nabla h(u(t)) =0,\ t\in [0, +\infty[,\ u(0)=x^0\in \R^n,\ \dot{u}(0)=\hat x^0\in \R^n,
\end{equation}
where $\alpha \geq 0$ induces a so-called viscous damping, and, respectively
\begin{equation}\label{DS-2}
  \ddot{u}(t) = \nabla h(u(t)) =0,\ t\in [0, +\infty[,\ u(0)=x^0\in \R^n,\ \dot{u}(0)=\hat x^0\in \R^n.
\end{equation}
Note that the investigations in \cite{RPK} are performed in Hilbert spaces, however we present their outcomes in the Euclidean framework considered everywhere in this survey.
Besides the absence of the viscous damping, \eqref{DS-2} differs from \eqref{DS2} by the sign of the gradient of $h$. Even though \eqref{DS-2} is thus simpler than \eqref{DS2}, its convergence statement is more demanding than the one of its counterpart. On the other hand, it delivers strong convergence toward the optimal solution to \eqref{COP}. They can be compared below. Before proceeding, let us note that, different to the considered first-order dynamical systems, in the current case the Cauchy-Lipschitz theorem only yields the existence of a solution $\widehat{u}\in C^2([0, +\infty), \R^n)$ to \eqref{DS2} and \eqref{DS-2} when $\nabla h$ is Lipschitz-continuous, however it is not necessarily unique (see \cite[Section 3]{RPK}.

\begin{theorem} (cf. \cite[Theorem 3.4]{RPK}) \label{th12}
When $\nabla h$ is Lipschitz-continuous on bounded sets and $\widehat{u}\in C^2([0, +\infty), \R^n)$ is bounded, $\liminf_{t\to +\infty}\overline{u}(t) = \overline{x}$ and $\liminf_{t\to +\infty}h(\overline{u}(t)) = h(\overline{x})$.
\end{theorem}

\begin{theorem} (cf. \cite[Theorem 25]{LMV}) \label{th13}
When there exists a $\kappa > 0$ such that for every trajectory $u\in C^2([0, +\infty), \R^n)$ of \eqref{DS2} one has
$$\nabla h(u(t))^\top (u(t)-\overline{x} \geq \kappa (h(u(t))-h(\overline{x})),$$
then $\liminf_{t\to +\infty}\overline{u}(t) = \overline{u}$ and $\liminf_{t\to +\infty}h(\overline{x}(t)) = h(\overline{x})$.
\end{theorem}

\begin{remark}
The convergence of $h(\overline{u}(t))$ toward $h(\overline{x})$ is not explicitly stated in \cite[Theorem 3.4]{RPK}, however it follows automatically as $h$ is continuous. On the other hand, the hypothesis of Theorem \ref{th13} is satisfied when $\nabla h$ is Lipschitz-continuous, see \cite{LMV}. Last but not least, note that the hypothesis $\argmin h\neq \emptyset$ from \cite[Theorem 3.4]{RPK} is not necessary, see \cite[Theorem 3.1]{GLT}, Theorem \ref{th2} and Remark \ref{re13}.
\end{remark}

Discretizing \eqref{DS2} with respect to time, one obtains the following heavy-ball algorithm for minimizing $h$.

\begin{algorithm}[H] (see \cite[$(66)$]{LMV})
\caption{Heavy-Ball Method for Strongly Quasiconvex Functions (HB-SQ)}\label{hb:sqcx}
\begin{description}
 \item[Step 0.] (Initialization). Let $\alpha, \eta>0$, $\theta = 1- \alpha\eta$, $x^0, x^1\in \R^n$ and $k = 1$.

 \item[Step 1.] If $\nabla h(x^k)=0_n$: \textsc{stop}.

\item[Step 2.] Compute
 $$ x^{k+1} = x^k  + \theta (x^k-x^{k-1}) -\eta^2 \nabla h (x^k) \qquad ~~~~~ [\text{inertial step}].
 $$

\item[Step 3.] Let $k = k+1$ and go to Step 1.
 \end{description}
\end{algorithm}

The counterpart of this algorithm that is provided in \cite[$(20)$]{RPK} as a time-discretization of \eqref{DS-2} is
\begin{algorithm}[H] (see \cite[$(22)$]{RPK})
\caption{Inertial Method for Strongly Quasiconvex Functions (IM-SQ)}\label{im:sqcx}
\begin{description}
 \item[Step 0.] (Initialization). Let $\eta>0$, $\{\alpha_j\}_j\subseteq \left[\eta,+\infty\right[$, $x^0, x^1\in \R^n$ and $k = 1$.

 \item[Step 1.] If $\nabla h(x^k)=0_n$: \textsc{stop}.

\item[Step 2.] Compute
 $$ x^{k+1} = 2 x^k  -x^{k-1} + \alpha_k \nabla h (x^k) \qquad ~~~~~ [\text{inertial step}].
 $$

\item[Step 3.] Let $k = k+1$ and go to Step 1.
 \end{description}
\end{algorithm}
As one can immediately notice, Algorithm \ref{im:sqcx} is on one level not so general as Algorithm \ref{hb:sqcx} because $2 x^k  -x^{k-1}$ is actually $x^k  + \theta (x^k-x^{k-1})$ for $\alpha = 0$ (and thus $\theta = 1$), while, on the other hand, it allows iterative coefficients to the gradient of $h$, that are constant in its counterpart.
Again, the convergence statement from \cite{RPK} reveals the (strong, in Hilbert spaces) convergence of the sequence generated by Algorithm \ref{im:sqcx} toward the unique minimizer of $h$. On the other hand, we removed its hypothesis $\argmin h\neq \emptyset$ that is not necessary, see \cite[Theorem 3.1]{GLT}, Theorem \ref{th2} and Remark \ref{re13}, and we added the convergence of the function values to the minimum value of $h$ (that follows automatically as $h$ is continuous).

\begin{theorem} (cf. \cite[Theorem 4.3]{RPK}) \label{th14}
When $\nabla h$ is Lipschitz-continuous on bounded sets, and the sequence $\{x^{k}\}_{k}$ generated by Algorithm \ref{im:sqcx} is bounded, then $\{x^{k}\}_{k}$ converges to $\overline{x}$ and $\{h(x^{k})\}_{k}$ to $h(\overline{x})$.
\end{theorem}

Different to the continuous case, the convergence statement for Algorithm \ref{hb:sqcx} from \cite{LMV} requires the additional hypothesis of Lipschitz-continuity on the gradient of $h$.

\begin{theorem} (cf. \cite[Theorem 27]{LMV}) \label{th15}
When $\nabla h$ is Lipschitz-continuous with constant $L$, $\theta \in ]0, 1[$, and $\eta^2\in ]0, (1-\theta^2)/L$, then the sequence $\{x^{k}\}_{k}$ generated by Algorithm \ref{hb:sqcx} converges to $\overline{x}$ and $\{h(x^{k})\}_{k}$ to $h(\overline{x})$.
\end{theorem}

\begin{remark}
Results on convergence rates of the trajectories of the considered dynamical systems can be found in \cite[Theorem 18]{LMV}, where it is shown that the trajectories of \eqref{DS1} in case $\psi\equiv 0$ converge exponentially to $\overline{x}$, and the same convergence order can be achieved for the values of $h$ on these trajectories toward its minimal value. Similar results are achieved for the second-order dynamical systems in \cite[Theorem 3.4]{RPK} and \cite[Theorem 25 $\&$ Remark 26]{LMV}. Moreover,
\cite[Remark 3.5]{RPK} exhibits the convergence to 0 of the gradient flow at an exponential rate.
On the other hand, in the discrete case, \cite[Theorem 22]{LMV} exhibits a linear convergence rate of the steepest descent algorithm for minimizing $h$ provided there, while \cite[Corollary 24]{LMV} provides (similar) convergence rates for the objective function values toward the minimal value of $h$. For the inertial methods, \cite[Theorem 4.3]{RPK} shows a convergence rate for the generated sequence, while \cite[Theorem 27 $\&$ Corollary 28]{LMV} provide (linear) convergence rates for both the generated sequence and the function values.
\end{remark}
\begin{remark}
Other iterative methods (of both proximal and extragradient type) for solving \eqref{COP} can be derived from the algorithms considered in Section \ref{sec:5} for solving
equilibrium problems governed by bifunctions taken strongly quasiconvex in their second variables. We leave this task to the interested reader.
\end{remark}

\section{Equilibrium problems involving strongly quasiconvex functions}\label{sec:5}


Let $K\subseteq {\R}^{n}$ be a nonempty, closed and convex set,
and a bifunction $f: K \times K \rightarrow {\R}$. Consider the \textit{equilibrium problem}
\begin{equation}\label{EP}
 {\rm find ~} \overline{x} \in K: ~~ f(\overline{x}, y) \geq 0, ~ \forall ~ y \in K.\tag{EP}
\end{equation}

Denote by $S(K, f)$ the solution set of \eqref{EP}. This class of problems provides an umbrella framework for various mathematical models from continuous optimization and variational analysis, including mi\-ni\-mi\-za\-tion problems, (inverse) variational inequalities, minimax problems, fixed point pro\-blems and complementarity pro\-blems among others (see \cite{BLO, KAR} for more on this).

The standard practice in the literature is to assume the convexity in the second variable of the bifunction employed in \eqref{EP}, and there are only few works there this hypothesis is relaxed, mostly because \eqref{EP} might fail to have solutions when this assumption is dropped, even when $K$ is convex and compact. Examples of equilibrium problems where the bifunction $f$ is taken to be only strongly quasiconvex in its second variable can be found in \cite[Example 4.2]{ILP}, \cite[Example 3]{MUY} and \cite[Examples 24--28]{GLT}, one of them being presented below, while examples of inverse mixed variational inequality that can be recast as an equilibrium problem governed by a bifunction that is strongly quasiconvex in its second variable are \cite[Example 4.1]{ILM} and \cite[Example 4.1]{LMY}.

\begin{example}\label{ex6} (see \cite[Example 27]{GLT}).
Take $p, q \in \mathbb{N}$, $p, q > 1$, and the
 bi\-func\-tion $f: {\R}^n \times {\R}^n \rightarrow {\R}$,
$$
f(x, y) \!=\! p \left( \max\{\sqrt{\| y\|}, (y - qe)^{2} \!-\! q\} \!-\!  \max\{\sqrt{\| x \|}, (x-qe)^{2} \!-\! q\} \right) + x^\top (y-x).$$
Then $f(x, \cdot)$ is strongly quasiconvex for all $x\in K$.
\end{example}

The following assumptions on the mathematical objects involved in the equi\-li\-brium problem are usually necessary for investigating and guaranteeing the convergence of iterative methods proposed for solving \eqref{EP}
\begin{itemize}
 \item[$(A1)$] $y\mapsto f(\cdot, y)$ is upper se\-mi\-con\-ti\-nuous for all $y \in K$;

 \item[$(A2)$] $f$ is pseudomonotone on $K$;

 \item[$(A3)$] $f$ is lower se\-mi\-con\-ti\-nuous (jointly in both arguments);

 \item[$(A4)$] $f(x, \cdot)$ is strongly quasiconvex on $K$ with modulus $\gamma > 0$ whenever $x \in K$;

 \item[$(A5)$] $f$ satisfies the following Lipschitz-continuity type condition:
 there exists $\eta> 0$ for which
 \begin{equation}\label{Lips:cond}
  f(x, z) - f(x, y) - f(y, z) \leq\eta(\left\lVert x - y \rVert^{2} + \lVert y - z \rVert^{2}\right), ~ \forall ~ x, y, z \in K.
 \end{equation}
\end{itemize}

\begin{remark}\label{re10}
Conditions $(A1)$, $(A2)$, $(A3)$ and $(A5)$ are standard assumptions in the literature on equilibrium problems, while $(A4)$ relaxes the standard convexity hypothesis
usually imposed on $f$ in its second argument. Recall also the following usual hypothesis when dealing with equilibrium problems that is a direct consequence of $(A2)$ and $(A5)$ (see, for instance, \cite[Remark 3.1$(i)$]{ILP}), and is mentioned here only for completeness (as it was employed for guaranteeing some technical results as a weaker assumption than $(A5)$, see \cite{GLT,ILP})
\begin{itemize}
 \item[$(A0)$] $f(x,x) = 0$ for all $x\in K$.
\end{itemize}
Moreover, a consequence of $(A3)$ that is sometimes employed instead of it is
\begin{itemize}
 \item[$(A1^{\prime})$] $f(x, \cdot)$ is lower se\-mi\-con\-ti\-nuous for all $x \in K$,
\end{itemize}
while for certain statements $(A3)$ is replaced with its stronger version
 \begin{itemize}
 \item[$(A3^{\prime})$] $f$ is continuous (jointly in both arguments) on an open set containing $K\times K$.
\end{itemize}
\end{remark}

An existence result for solutions to  \eqref{EP} can be found in \cite[Theorem 3.1]{IL3}, based on the $q$-asymptotic function (see \cite{FFB-Vera,IL2,lara-lopez}). Adding to it the strong quasiconvexity in the second variable of the involved bifunction one can guarantee the uniqueness of the solution to \eqref{EP}.

\begin{proposition} (cf. \cite[Proposition 3.1]{ILP}) \label{pr24}
When $f$ satisfies hypotheses $(Ai)$ with $i=0,1,1^{\prime},2,4$, then $S(K, f)$ is a singleton.
\end{proposition}

The first proximal point algorithm for problem \eqref{EP} was presented in \cite{IL5} (in its classical form), here we present the
relaxed-inertial proximal point algorithm considered in \cite{GLT}, which was proposed in the convex framework in \cite{HDT, VTV}, and is inspired by Algorithm \ref{rippa:sqcx}.

\begin{algorithm}[H] (see \cite[Algorithm 1]{GLT})
 \caption{RIPPA-EP for Strongly Quasiconvex EP's}\label{rippa:ep}
\begin{description}
 \item[Step 0.](Initialization). Let  $x^0, x^{-1} \in K$, $\alpha, \rho \in
 [0, 1[$ and  $\{\beta_{j}\}_{j} \subseteq ]0, + \infty[$, $k=0$.

 \item[Step 1.] Choose $\alpha_k\in [0,\alpha]$, set
 $$ y^{k} = x^{k} + \alpha_{k} (x^{k} - x^{k-1}), \qquad [\text{extrapolation step}]$$
 and compute
 $$ z^{k} \in \argmin_{x \in K} \left\{f(y^{k}, x) + \frac{1}{2 \beta_{k}} \lVert
  y^{k} - x \rVert^{2}\right\}. \qquad [\text{proximal step}]
 $$

 \item[Step 2.] If $y^{k} = z^{k}$: \textsc{stop}.

 \item[Step 3.] Choose some relaxation parameter {$\rho_k \in [1-\rho,1+\rho]$}, and update
 $$
    x^{k+1}=(1-\rho_k)y^k+\rho_k z^{k}. \qquad [\text{relaxation step}]
 $$

 \item[Step 4.] Let $k\leftarrow k+1$ and go to Step 1.
 \end{description}
\end{algorithm}

\begin{remark}
When the stopping criterion in Algorithm \ref{rippa:ep} $y^{k} = z^{k}$ is activated, the solution to \eqref{EP} has been identified, i.e. $S(K, f)=\{y^k\}$.
\end{remark}

For achieving the convergence of Algorithm \ref{rippa:ep}, the following additional hypotheses on the involved parameter sequences were proposed in \cite{GLT}
$$\exists \, \epsilon > 0:
\left\{
 \begin{array}{ll}
(C1) & \frac{1}{\gamma - 8 \eta} < \beta_{k} < \epsilon \leq
  \frac{1}{4 \eta}\ \forall k \geq 0;\\

(C2) & 0 < 1 - \rho \leq \rho_{k} \leq 1 + \rho:\ 0 \leq \rho
  \leq 1 - 4 \eta \epsilon.
\end{array}
\right.$$
In \cite[Remark 8]{GLT} their meaning is related to the existing literature. The convergence result of Algorithm \ref{rippa:ep} follows (for preliminary technical results see \cite[Subsection 3.1]{GLT}.

\begin{remark}
 Assumption $(C1)$ suggests that one needs to have
 \begin{equation}\label{gamma:eta}
  \gamma > 12 \eta.
 \end{equation}
This assumption was used in \cite{GLT,ILP,ILM,LMY} but it follows from the analysis of \cite[Proposition 3.4]{ILP}, in which can be noted that
 $$\left( \frac{1 + \gamma \beta_{k}}{8 \beta_{k}} - \eta \right) \geq 0  \Longleftrightarrow ~ 1 \geq \beta_{k} (8 \eta - \gamma) ~ \forall ~ k \in \mathbb{N}.$$

 Condition \eqref{gamma:eta} is, thus, correct, but it is not complete since it only deals with the case when $\gamma \geq 8\eta$. It can be improved to the case
 $$ \gamma \in \, ]0, 8 \eta[ \, \cup \, ]12 \eta, + \infty[,$$ 
when the parameters $\{\beta_{k}\}_k$ could be taken as follows
 \begin{align}
  & {\rm if ~}~ 0 < \gamma < 8 \eta ~ \Longrightarrow ~
  \beta_{k} \in \left] 0, \min\left\{\frac{1}{8 \eta - \gamma}, \frac{1}{4 \eta}\right\} \right[ \label{new:c1} \\
  & {\rm if ~} ~~~~~ \gamma > 12 \eta ~ \Longrightarrow ~ \beta_{k} \in \left] \frac{1}{\gamma - 8 \eta}, \frac{1}{4 \eta} \right[. \label{new:c2}
 \end{align}
\end{remark}
Hence, we can replace assumptions $(C1)$ and \eqref{gamma:eta} by relations \eqref{new:c1} and \eqref{new:c2} in the next statements.

\begin{theorem}\label{conver:solution} (see \cite[Theorem 11]{GLT})
 Let $K$ be an affine subspace in ${\R}^{n}$, $f$ be such that assumptions
 $(Ai)$ with $i = 1, 2,3,4,5$ hold, $\{\beta_{k}\}_{k}$ and $\{\rho_{k}\}_{k}$
 be sequences of po\-si\-ti\-ve num\-bers such that assumptions $(Ci)$ with
 $i=1,2$ hold, $\{x^{k}\}_{k}$, $\{y^{k}\}_{k}$ and $\{z^{k}\}_{k}$ be the
 sequences generated by Algorithm \ref{rippa:ep}. If
 \begin{align}\label{eq:th:main.09}
  \sum_{k=0}^\infty \, \alpha_{k} \norm{x^{k} - x^{k-1}}^{2} < + \infty,
  \end{align}
 then the following assertions hold
 \begin{itemize}
  \item [$(a)$] for $\{\overline{x}\} = S(K, f)$, the limit $\lim_{k \to
  \infty}\,\norm{x^k-\overline{x}}$ exists and
  \begin{align}\label{eq:16}
   \lim_{k \to + \infty} \lVert x^{k+1} - y^{k} \lVert^{2} = \lim_{k \to +
   \infty} \lVert z^{k} - y^{k} \lVert^{2} = 0;
  \end{align}

  \item [$(b)$] the sequences $\{x^{k}\}_{k}$, $\{y^{k}\}_{k}$ and $\{z^{k}\}_{k}$ converge all to $\overline x$.
 \end{itemize}
\end{theorem}

\begin{remark}
Similarly to Theorem \ref{th1} (see Proposition \ref{pr22}), one can show (see \cite[Theorem 18]{GLT}) that the hypothesis \eqref{eq:th:main.09} of Theorem \ref{conver:solution} is valid under the  condition (see also \cite[Remark 12]{GLT} for another one)
\begin{itemize}
 \item[$(C3)$] the sequence $\{\alpha_k\}_{k}$ is nondecreasing and there
 exists $\beta = {\xi}/({2 + \xi})$ such that \eqref{eq:alpha} holds
where
\begin{align}\label{lower.b}
\xi:=\frac{1-4\eta \epsilon-\rho}{1+\rho}>0.
\end{align}
\end{itemize}
\end{remark}

\begin{remark}\label{re7}
When $\alpha = 0$ it immediately follows $\alpha_{k} = 0$ for all $k \in \N$, and Algorithm \ref{rippa:ep} collapses into a relaxed proximal point algorithm for solving\eqref{EP} (see \cite[Corollary 12]{GLT}) where $K$ can be taken only closed and convex (see also Remark \ref{re4}) when $\rho_{k} \in \, ]0, 1]$ for every $k \in \N$ (or an affine subspace when $\rho_{k} \in \, ]0, 2[$, $k \in \N$).
\end{remark}

\begin{remark}\label{re8}
When besides $\alpha = 0$ one also fixes $\rho_{k} = 1$ for all $k \in \N$, Algorithm \ref{rippa:ep} turns into the proximal point method for solving \eqref{EP} (see \cite[Algorithm 1]{ILP} and \cite[Corollary 14]{GLT}).
\end{remark}

\begin{remark}\label{re9}
In \cite[Subsection 4.2]{GLT} it is shown that a suitable choice of the relaxation and inertial parameters ensures that Algorithm \ref{rippa:ep} solves \eqref{EP} faster and cheaper than its standard proximal point counterpart \cite[Algorithm 1]{ILP}.
\end{remark}

Next we present other developments of the proximal point method (besides the relaxed-inertial Algorithm \ref{rippa:ep}) proposed for solving \eqref{EP} when the employed bifunction is strongly quasiconvex in the second variable.

We begin with the one proposed in \cite{ILP} (following the method introduced in \cite{IUS} in the convex case) where one regularizes the whole bifunction $f$ instead of applying this procedure only on its second variable. Even though the convergence statement of this algorithm holds under less demanding hypotheses than the one of Algorithm \ref{rippa:ep}, this method is not really practical because it involves solving a regularized equilibrium problem in each iteration.

\begin{algorithm}[H] (see \cite[Algorithm 2]{ILP})
 \caption{REG-EP for Strongly Quasiconvex EP's}\label{reg:ep}
\begin{description}
 \item[Step 0.](Initialization). Let  $x^0 \in K$, $\theta>0$, $\{\beta_{j}\}_{j \in\mathbb{N}} \subseteq [\theta, + \infty[$, $k=0$.

 \item[Step 1.] Set
 $$
  f_k(x, y) = f(x, y) + \frac{1}{\beta_k} (x-x^{k})^\top (y- x), \qquad [\text{regularization step}]
 $$
 and compute
 $$
  x^{k} \in S(K, f_k).
 $$

 \item[Step 2.] If $x^{k+1} = x^{k}$: \textsc{stop}.

 \item[Step 3.] Let $k\leftarrow k+1$ and go to Step 1.
 \end{description}
\end{algorithm}

\begin{remark}
When the stopping criterion in Algorithm \ref{reg:ep} $x^{k+1} = x^{k}$ is activated, the solution to \eqref{EP} has been identified, i.e. $S(K, f)=\{x^k\}$.
\end{remark}

The convergence statement follows immediately, partially as a consequence of the technical results from \cite[Proposition 3.6]{ILP}.

\begin{theorem} (see \cite[Corollary 3.2]{ILP})\label{th5}
Let $f$ be such that assumptions $(Ai)$ with $i = 1,1^{\prime}, 2,4,5$ hold, and $\{x^{k}\}_{k}$ be the
sequence generated by Algorithm \ref{reg:ep}. Then $\{x^{k}\}_{k}$ converges to $\overline x$, where $S(K, f)=\{\overline{x}\}$.
\end{theorem}

\begin{remark}
The original convergence result for Algorithm \ref{reg:ep} given in \cite[Corollary 3.2]{ILP} also contains the hypothesis $(A3)$ (alongside its weaker version $(A1^{\prime})$) which does not seem to be necessary for guaranteeing the outcome.
\end{remark}

\begin{remark}\label{re11}
As noted in \cite[Corollary 3.2]{ILP}, one can derive another algorithm for solving \eqref{COP} (for $h: K \rightarrow \mathbb{R}$) from Algorithm \ref{reg:ep} by taking $f(x,y)=h(y)-h(x)$, $x, y\in K$. The hypotheses of Theorem \ref{th5} are fulfilled when $h$ is lower semicontinuous and strongly quasiconvex with modulus $\gamma > 0$, and the sequence generated by the algorithm converges toward the minimizer of $h$.
\end{remark}

The second algorithm for solving \eqref{EP} developed on the basis of the proximal point one we present here is the inertial extrapolation one that comes from \cite{IOS}. Different to the inertial proximal point method that can be derived from Algorithm \ref{rippa:ep} by canceling the relaxation steps, this method allows the inertial parameter (note that it involves a constant one, i.e. the counterpart of $\{\alpha_j\}_j$ from Algorithm \ref{rippa:ep} is a constant sequence) to take negative values, too.

\begin{algorithm}[H] (see \cite[Algorithm 1]{IOS})
 \caption{IEPPA-EP for Strongly Quasiconvex EP's}\label{ieppa:ep}
\begin{description}
 \item[Step 0.](Initialization). Let  $x^0, y^{0} \in \R^n$, $\alpha \in
]-1, 1[$, $\beta >0$, $k=0$.

 \item[Step 1.] Determine
$$
  x^{k+1} \in \argmin_{x \in K} \left\{f(y^{k}, x) + \frac{1}{2 \beta} \lVert
  y^{k} - x \rVert^{2}\right\}. \qquad [\text{proximal step}]
$$

 \item[Step 2.] If $x^{k+1} = y^{k}$: \textsc{stop}.

 \item[Step 3.] Update
$$
  y^{k+1} = x^{k+1} + \alpha (x^{k+1} - x^{k}). \qquad [\text{extrapolation step}]
$$

\item[Step 4.] Let $k\leftarrow k+1$ and go to Step 1.
 \end{description}
\end{algorithm}

\begin{remark} (see \cite[Proposition 3.6]{IOS})
When the stopping criterion in Algorithm \ref{ieppa:ep} $x^{k+1} = y^{k}$ is activated, the solution to \eqref{EP} has been identified, i.e. $S(K, f)=\{x^{k+1}\}$.
\end{remark}

\begin{remark} (see \cite[Remark 3.4]{IOS})
When $\alpha = 0$ Algorithm \ref{ieppa:ep} turns into the proximal point method for solving \eqref{EP} that is \cite[Algorithm 1]{ILP}.
\end{remark}

The convergence of Algorithm \ref{ieppa:ep} is discussed separately when the inertial parameter $\alpha$ is positive or negative. Technical statements such as \cite[Lemma 3.8, Lemma 3.9 \& Lemma 3.14]{IOS} provided important tools for the proofs.

\begin{theorem}\label{th6} (see \cite[Theorem 3.10]{IOS})
Let $f$ be such that assumptions $(Ai)$ with $i = 1, 2,3,4,5$ hold, $12\eta <\gamma$, $0\leq \alpha <1/3$, $1/(\gamma - 8 \eta)<\beta < (1-3\alpha)/(4\eta(1-\alpha))$,
and $\{x^{k}\}_{k}$ be the sequence generated by Algorithm \ref{ieppa:ep}. Then $\{x^{k}\}_{k}$ converges to $\overline x$, where $\{\overline{x}\} = S(K, f)$.
\end{theorem}

\begin{theorem}\label{th7} (see \cite[Theorem 3.14]{IOS})
Let $f$ be such that assumptions $(Ai)$ with $i = 1, 2,3,4,5$ hold, $12\eta <\gamma$, $-1< \alpha <0$, $1/(\gamma - 8 \eta)<\beta < 1/(4\eta)$,
and $\{x^{k}\}_{k}$ be the sequence generated by Algorithm \ref{ieppa:ep}. Then $\{x^{k}\}_{k}$ converges to $\overline x$, where $\{\overline{x}\} = S(K, f)$.
\end{theorem}

Last but not least, we recall the two-step predictor-corrector proximal point algorithm proposed in \cite{ILM} for solving \eqref{EP}. In the original source it is available in Hilbert spaces but we consider it here in the setting considered in the rest of the survey, namely finitely dimensional spaces. Before formulating the mentioned algorithm, let us note that \cite[Lemma 2]{MUY} (see also \cite{ILM, LMY} reveals the coincidence of the solution set to \eqref{EP} and its \textit{dual (Minty) equilibrium problem}
\begin{equation}\label{DEP}
 {\rm find ~} \overline{z} \in K: ~~ f(y, \overline{z}) \leq 0, ~ \forall ~ y  \in K,\tag{DEP}
\end{equation}
when the governing bifunction is strongly quasiconvex in the second variable and pseudomonotone. In general, the solution set of \eqref{EP} is included in the one of \eqref{DEP} whenever $(A2)$ holds, while the reverse inclusion holds when $(A4)$ is fulfilled (more generally, when $f$ is semistrictly quasiconvex in the second variable).

\begin{proposition} (cf. \cite[Remark 3.1$(i)$]{ILM}) \label{pr25}
When $f$ satisfies hypotheses $(Ai)$ with $i=0,2,3,4$, then the solution sets of \eqref{EP} and \eqref{DEP} coincide, and $S(K, f)$ is a singleton.
\end{proposition}

Different to the other proximal point type methods for solving \eqref{EP} presented above, this method includes two proximal steps, being thus a predictor-corrector method, with a first prediction step and a second correction one, in which the objective function is updated while trying to remain close enough to the previous iterate.

\begin{algorithm}[H] (see \cite[Algorithm 1]{ILM})
 \caption{2PPA-EP for Strongly Quasiconvex EP's}\label{2ppa:ep}
\begin{description}
 \item[Step 0.](Initialization). Let  $x^0\in K$, $\varepsilon >0$, $\{\beta_j\}_j \subseteq [1/(\gamma-8\eta) + \varepsilon, 1/(e\eta)-\varepsilon]$, $k=0$.

 \item[Step 1.] Determine
$$
  y^{k} \in \argmin_{x \in K} \left\{f(x^{k}, x) + \frac{1}{2 \beta_k} \lVert
  x^{k} - x \rVert^{2}\right\}. \qquad [\text{predictor proximal step}]
$$

 \item[Step 2.] If $y^{k} = x^{k}$: \textsc{stop}.

 \item[Step 3.] Determine
\begin{equation}\label{2p}
  x^{k+1} \in \argmin_{x \in K} \left\{f(y^{k}, x) + \frac{1}{2 \beta_k} \lVert
  x^{k} - x \rVert^{2}\right\}. \qquad [\text{corrector proximal step}]
\end{equation}

\item[Step 4.] Let $k\leftarrow k+1$ and go to Step 1.
 \end{description}
\end{algorithm}

\begin{remark} (see \cite[Proposition 3.1]{ILM})
When the stopping criterion in Algorithm \ref{2ppa:ep} $x^{k} = y^{k}$ is activated, the solution to \eqref{EP} has been identified, i.e. $S(K, f)=\{x^{k}\}$.
\end{remark}

\begin{remark} (see \cite[Corollary 3.2]{ILM})
When \eqref{2p} is replaced by $x^{k+1}=y^k$, then Algorithm \ref{2ppa:ep} turns into the proximal point method for solving \eqref{EP} that is \cite[Algorithm 1]{ILP}.
\end{remark}

The technical statements \cite[Propositions 3.1--3.3 \& Theorem 3.1]{ILM} provided important tools for the convergence proof of Algorithm \ref{2ppa:ep}, presented below in a simplified form. Note that in \cite[Theorem 3.1]{ILM} the convergence of the sequences generated by Algorithm \ref{2ppa:ep} is strong (recall that the result is provided there in Hilbert spaces).

\begin{theorem}\label{th8} (see \cite[Theorem 3.1]{ILM})
Let $f$ be such that assumptions $(Ai)$ with $i = 1, 2,3,4,5$ hold, $12\eta <\gamma$, $\varepsilon >0$, $\{\beta_j\}_j \subseteq [1/(\gamma-8\eta) + \varepsilon, 1/(e\eta)-\varepsilon]$, and $\{x^{k}\}_{k}$ and $\{y^{k}\}_{k}$ be the sequences generated by Algorithm \ref{2ppa:ep}. Then, denoting $\{\overline{x}\} = S(K, f)$, one has
$$\lim_{k \to \infty} \lVert x^{k} - \overline{x} \rVert=\lim_{k \to \infty} \lVert x^{k+1} - y^k \rVert=\lim_{k \to \infty} \lVert x^{k} - y^k \rVert =0,$$
and both $\{x^{k}\}_{k}$ and $\{y^{k}\}_{k}$ converge to $\overline x$.
\end{theorem}

In the recent literature one can also find two extragradient methods for solving equilibrium problems governed by bifunctions that are strongly quasiconvex in their second variables, see \cite{MUY, LMY}. As stressed in \cite[Remark 4.1]{LMY}, the Lipschitz-continuity type assumption $(A5)$ on $f$ is not necessary for guaranteeing the convergence of these algorithms. We begin with the one proposed in \cite{MUY}.

\begin{algorithm}[H] (see \cite[Algorithm 3.1]{MUY})
 \caption{EG-EP for Strongly Quasiconvex EP's}\label{eg:ep}
\begin{description}
 \item[Step 0.](Initialization). Let  $x^0\in K$, $k=0$, $\alpha, \rho \in ]0, 1[0$, $\{\beta_j\}_j, \{\alpha_j\}_j \subseteq ]0, +\infty[$ such that
 $\{\beta_j\}_j$ nonincreasingly converges to some $\beta>0$,
 $\sum_{j=0}^{+\infty} \alpha_j = +\infty$ and $\sum_{j=0}^{+\infty} \alpha_j^2 < +\infty$.

 \item[Step 1.] Determine
$$
  y^{k} \in \argmin_{x \in K} \left\{f(x^{k}, x) + \frac{1}{2 \beta_k} \lVert x^{k} - x \rVert^{2}\right\}. \qquad [\text{proximal step}]
$$

 \item[Step 2.] If $y^{k} = x^{k}$: \textsc{stop}.

 \item[Step 3.] Determine the smallest $m\in \N$ such that
 $$ f(z^k, x^k) - f(z^k, y^k) \geq \frac {\alpha}{2\beta_k} \lVert x^{k} - y^k \rVert^{2},$$
where $z^k:=(1-\rho^m)x^k + \rho^my^k$.

\item[Step 4.] Choose
$$w^k\in \partial^{*}f(z^k, \cdot) (x^k) = \left\{w\in \R^n: w^\top (y-x^k) < 0 \text{ if }f(z^k, y) < f(z^k, x^k)\right\}  \qquad [\text{subgradient step}]$$ and compute
$$x^{k+1} = \pr_K \left(x^k - \frac{\alpha_k}{\|w^k\|} w^k\right).  \qquad [\text{projection step}]$$

\item[Step 5.] If $x^{k} = x^{k+1}$: \textsc{stop}.

\item[Step 6.] Let $k\leftarrow k+1$ and go to Step 1.
 \end{description}
\end{algorithm}

\begin{remark} (see \cite[Proposition 1$(ii)$ \& Proposition 3]{MUY})
When one of the stopping criteria in Algorithm \ref{eg:ep} is activated, the solution to \eqref{EP} has been identified. More precisely, $x^{k} = y^{k}$ yields $S(K, f)=\{x^{k}\}$, while $x^{k} = x^{k+1}$ guarantees that $S(K, f)=\{z^{k}\}$.
\end{remark}

\begin{remark}
Algorithm \ref{eg:ep} can be seen, too, as a development of the proximal point method for solving \eqref{EP} (see \cite[Algorithm 1]{ILP}) as it adds a subgradient and a projection step to the proximal one.
\end{remark}

The technical statements \cite[Propositions 1--4]{MUY} provide important tools for the convergence proof of Algorithm \ref{eg:ep}, presented below in a simplified form. Note the more involved hypotheses required in this case for algorithm convergence, when one compares them with the ones needed for other similar statements. See also \cite[Remark 2]{MUY} for some comments on this and alternate hypotheses. 

\begin{theorem}\label{th9} (see \cite[Theorem 1]{MUY})
Let $f$ be such that assumptions $(A3^{\prime})$ and $(A4)$ hold, and the solution set of \eqref{DEP} is nonempty. Let $\{x^{k}\}_{k}$ and $\{y^{k}\}_{k}$ be the sequences generated by Algorithm \ref{eg:ep}. Then, denoting $\{\overline{x}\} = S(K, f)$, if $\{y^{k}\}_{k}$ is bounded, the sequence $\{x^{k}\}_{k}$ converges to $\overline x$.
\end{theorem}

Last but not least we present the extragradient projection method for solving strongly quasiconvex equilibrium problems proposed in \cite{LMY}, where the strong subdifferential is employed in the subgradient step instead of the star one used above (see \cite[Remark 3.3]{LMY} for a discussion of the advantages of this choice).

\begin{algorithm}[H] (see \cite[Algorithm 1]{LMY})
 \caption{PEG-EP for Strongly Quasiconvex EP's}\label{peg:ep}
\begin{description}
 \item[Step 0.](Initialization). Let  $x^0\in K$, $k=0$, $\alpha, \rho \in ]0, 1[0$, $\{\beta_j\}_j, \{\alpha_j\}_j \subseteq ]0, +\infty[$.

 \item[Step 1.] Determine
$$
  y^{k} \in \argmin_{x \in K} \left\{f(x^{k}, x) + \frac{1}{2 \beta_k} \lVert x^{k} - x \rVert^{2}\right\}. \qquad [\text{proximal step}]
$$

 \item[Step 2.] If $y^{k} = x^{k}$: \textsc{stop}.

 \item[Step 3.] Determine the smallest $m\in \N$ such that
 $$ f(z^k, x^k) - f(z^k, y^k) \geq \frac {\alpha}{2\beta_k} \lVert x^{k} - y^k \rVert^{2},$$
where $z^k:=(1-\rho^m)x^k + \rho^my^k$.

\item[Step 4.] Choose
$$w^k\in \partial^{K}_{\beta_k, \gamma}f(z^k, \cdot) (x^k) = \left\{w\in \R^n: w^\top (y-x^k) < 0 \text{ if } f(z^k, y) < f(z^k, x^k)\right\}  \qquad [\text{subgradient step}]$$ and compute
$$x^{k+1} = \pr_K \left(x^k - {\alpha_k} w^k\right).  \qquad [\text{projection step}]$$

\item[Step 5.] If $x^{k} = x^{k+1}$: \textsc{stop}.

\item[Step 6.] Let $k\leftarrow k+1$ and go to Step 1.
 \end{description}
\end{algorithm}

\begin{remark} (see \cite[Remark 3.4 \& Proposition 3.2]{LMY})
When one of the stopping criteria in Algorithm \ref{peg:ep} is activated, the solution to \eqref{EP} has been identified. More precisely, $x^{k} = y^{k}$ yields $S(K, f)=\{x^{k}\}$, while $x^{k} = x^{k+1}$ guarantees that $S(K, f)=\{z^{k}\}$.
\end{remark}

\begin{remark}
Algorithm \ref{peg:ep} can be seen, too, as a development of the proximal point method for solving \eqref{EP} (\cite[Algorithm 1]{ILP}) as it adds a subgradient and a projection step to the proximal one.
\end{remark}

The technical statements \cite[Propositions 3.1--3.5 \& 4.2]{LMY} provided important tools for the convergence proof of Algorithm \ref{peg:ep}, given below.

\begin{theorem}\label{th16} (see \cite[Theorem 3.1 $\&$ Proposition 3.6]{LMY})
Let $f$ be such that assumptions $(Ai)$ with $i = 2,3^{\prime},4$ hold, $K\subseteq \inte \dom f(x, \cdot)$ for all $x\in K$, and  denote $\{\overline{x}\} = S(K, f)$.
Let $\{x^{k}\}_{k}$, $\{y^{k}\}_{k}$ and $\{z^{k}\}_{k}$ be the sequences generated by Algorithm \ref{peg:ep} when the sequences $\{\beta_j\}_j, \{\alpha_j\}_j \subseteq ]0, +\infty[$ fulfill
$0<\alpha_j< \inf_{k\in \N} \beta_k$ for all $j\in \N$, $\sum_{j=0}^{+\infty} \alpha_j = +\infty$ and $\sum_{j=0}^{+\infty} \alpha_j^2 < +\infty$.
If there exists an $M>0$ such that for every $x\in K$ and every $k\in \N$ one has $\partial^{K}_{\beta_k, \gamma}f(x, \cdot) (x^k)\subseteq {\cal B} (0_n, M)$,
then $\{x^{k}\}_{k}$, $\{y^{k}\}_{k}$ and $\{z^{k}\}_{k}$ are bounded and converge to $\overline x$.
\end{theorem}

\begin{remark}
Results on convergence rates of the iterative methods mentioned above for solving \eqref{EP} can be found in \cite[Theorem 3.11, Theorem 3.15 $\&$ Theorem 3.16]{IOS}, \cite[Corollary 3.1]{LMY}, and, in Hilbert spaces, in \cite[Corollaries 3.1--3.2]{ILM}. It should be noted, however, that the hypotheses of \cite[Theorem 3.16]{IOS} seem to be contradictory as they ask both $\alpha <0$ and $\alpha =0$.
Note also \cite[Corollary 20]{GLT} where the number of iterations necessary for Algorithm \ref{rippa:ep} to reach a suitable approximation of the solution of \eqref{EP} is provided.
\end{remark}

\begin{remark}
Numerical experiments showing the usability and performance of the iterative methods for solving \eqref{EP} can be found in most of the papers mentioned in this section.
In \cite[Section 4]{MUY} one can see how is Algorithm \ref{eg:ep} performing on some simple equilibrium problems. In \cite[Section 4.2]{ILM} it is discussed that Algorithm \ref{2ppa:ep} outperforms both the proximal point method for solving \eqref{EP} given in \cite[Algorithm 1]{ILP} and the extragradient method Algorithm \ref{eg:ep} when dealing with some concrete equilibrium problems governed by bifunctions that are strongly quasiconvex in their second variables. In \cite[Remark 4.1]{IOS} it is discussed that for a suitable choice of the inertial parameter (namely when it approaches $0$ from any side) ensures that Algorithm \ref{ieppa:ep} outperforms its standard proximal point counterpart \cite[Algorithm 1]{ILP} in the task of solving \eqref{EP}. Similarly, in \cite[Section 4.2]{GLT} it is shown that Algorithm \ref{rippa:ep} outperforms \cite[Algorithm 1]{ILP} for certain constellations of parameters, while in \cite[Section 4.2]{LMY} the same is stressed for Algorithm \ref{peg:ep} in relation with Algorithm \ref{eg:ep}
\end{remark}

\begin{remark}
While no dynamical systems for asymptotically approaching the solution to \eqref{EP} are known to us (albeit continuous versions of the algorithms presented in this section are likely to be possible to formulate and study), in \cite{CGC} one can find a study on asymptotic convergence properties of saddle-point dynamics governed by $C^1$ bifunctions, where some results are formulated under strong quasiconvexity hypotheses. More precisely, in \cite[Proposition 4.8]{CGC} one finds results on local asymptotic stability of the set of saddle points of the considered bifunction.
\end{remark}

\section{Future developments}\label{sec:6}

In the following we discuss some ideas and possible directions for further studying strongly quasiconvex functions and optimization problems involving them.

Further theoretical investigations into the properties of strongly quasiconvex functions are always welcome. In particular, as all such functions known so far (see Subsection \ref{su31}) are continuous, a legitimate question is whether the real-valued strongly quasiconvex functions defined over $\R^n$ are continuous (like their convex counterparts). Of course, identifying other examples of strongly quasiconvex functions (with practical relevance) would contribute to advances in their study as well.

While many properties of the strong subdifferential (of a strongly quasiconvex function) have already been established (in \cite{KAL,LMC,LY}), at least two aspects regarding them could have interesting implications. One concerns calculus properties (in particular, the identification of hypotheses guaranteeing some inclusion relation between the strong subdifferential of a sum of functions and the sum of the strong subdifferentials of the involved functions) and the other computational methods for determining (or approximating) it. For instance, the latter would yield the practicability of the subgradient algorithm proposed in \cite{LMC} for minimizing a strongly quasiconvex function over a closed convex set. Moreover, other optimality conditions (in terms of the strong subdifferential but not only) for optimization problems involving strongly quasiconvex functions that exploit the specificities of the latter would surely be welcome.

While in general the proximity operator of a strongly quasiconvex function is set-valued, there are examples when it is a single-valued mapping (see Remark \ref{re1}). We would be interested in identifying the (or, at least a) class of strongly quasiconvex functions which remain strongly quasiconvex when added to a half of a squared norm (possibly multiplied by a positive constant). A follow-up question to this issue is a possible connection to prox-convex functions, more precisely, is this sought class maybe the intersection of the sets of strongly quasiconvex and prox-convex functions?

Proximity operators of strongly quasiconvex functions are known only in some isolated cases (e.g. when the functions are also convex), and closed forms for proximity operators on closed convex subsets are mostly not available even for convex functions. Possible ideas for providing progress in this direction could use some hints from \cite{BRL} (where the proximity operator of the root function is determined) or \cite{GSV} (where the proximity operator of a convex function over a nonconvex set is computed). Other related contributions are \cite{HSA} (a method based on the computation of proximal points of piecewise affine models of certain classes of nonconvex functions is proposed for determining the proximal points of the latter), \cite{LLT} (a symbolic computation method) and \cite{FGO} (an interior method). Moreover, determining the Bregman proximity operator of a strongly quasiconvex function on a certain closed convex set would render the Bregman proximal point method for minimizing strongly quasiconvex functions from \cite{LAM} usable in practice.

It also remains an open question if there are choices of the inertial and relaxation parameters of the relaxed-inertial proximal point algorithms for minimizing strongly quasiconvex functions (over closed convex sets or linear subspaces) and for solving equilibrium problems involving bifunctions strongly quasiconvex in the second variable (see \cite{GLM, GLT}) that guarantee an acceleration of the standard proximal point method (similar to the situation the convex framework studied in \cite{ATC}, see also Remark \ref{re2}). Improvements on the hypotheses required for establishing the convergence of the proximal point methods for solving optimization problems involving strongly quasiconvex functions would be welcome as well, for instance with respect to the Lipschitz-type condition $(A5)$.

Extending the applicability of other (proximal point type) algorithms from the convex framework to optimization problems involving strongly quasiconvex functions is another topic for subsequent research. Related to it, investigations on whether other algorithms that are accelerated (with respect to the convex setting) when the involved function is strongly convex remain accelerated when the latter is strongly quasiconvex would be of interest, too, in the light of Remark \ref{re14}.

Regarding the dynamical systems governed by gradients of strongly quasiconvex functions, a legitimate question concerns the necessity of imposing boundedness hypotheses on their trajectories in order to ensure their convergence to the minimizers of the considered functions in \cite{ROP, RPK}. Related to it, do the sequences generated by the algorithms derived from these dynamical systems via time-discretization have to be bounded in order to converge to the minimizers of the involved functions?

Worth exploring seem to be also splitting type (proximal point) algorithms for minimizing sums involving strongly quasiconvex functions. A first step in this direction was given in \cite{LY}, in which the authors developed the proximal gradient method for minimizing the sum of two nonconvex functions, one nonsmooth and strongly quasiconvex, and the other  differentiable with Lipschitz-continuous gradient (and possibly nonconvex, too), by using the strong subdifferential.
Similarly, algorithms for minimizing differences of convex functions might be adaptable for minimizing differences of functions involving at least one that is strongly quasiconvex. Furthermore, extending the usability of some algorithms from strongly quasiconvex functions to (other classes of) quasiconvex ones (as done in \cite{LAR, GLM}) is another path whose pursuing could have important practical consequences, given the various applications involving quasiconvex functions (for instance in economics).

Last but not least, taking into consideration that first order methods such as proximal point and gradient type ones have been proven to deliver minima for both convex and strongly quasiconvex functions, even though each of these classes contains elements that do not belong to the other, we are wondering: {\it
which is the largest class of functions (including, in particular, both convex and strongly quasiconvex ones) for which first order methods exhibit a similar behavior?}
This will be the matter of some of our subsequent work.

\subsection{Availability of supporting data}

No data sets were generated during the current study. The used
{\sc matlab} codes are available from all authors on reasonable request.

\subsection{Competing interests}

There are no conflicts of interest or competing interests related to this
manuscript.

\subsection{Funding}

The authors would like to thank the MATH AmSud cooperation program
(Pro\-ject AMSUD-220020) for its support. This research was partially
supported by Anid--Chile under project Fondecyt Regular 1241040
(Lara), by a public grant as part of the Investissement d'avenir project, reference
ANR-11-LABX-0056-LMH, LabEx LMH (Grad), and by BASAL fund FB210005 for center of excellence from ANID-Chile (Marcavillaca).

\subsection{Authors' contributions}

All authors contributed equally to the study conception, design and implementation, and wrote and corrected the manuscript.

\subsection{Acknowledgments}

The authors are grateful to the Guest Editors for the invitation to contribute to this Special Issue with a survey paper, to Alireza Kabgani and Phan Tu Vuong for several suggestions that led to improvements in the presentation of the strong subdifferential and for the joint works \cite{GNV, LMV}, respectively, and to other colleagues for their feedback and questions on issues related to strongly quasiconvex functions at various conferences and workshops where we presented results involving strongly quasiconvex functions. Finally, we would like to dedicate this work also to our former Master's student MSc. Juan Choque, who passed away on August 13, 2024, and whose only scientific publication \cite{LMC} deals with strongly quasiconvex functions.

\end{document}